\documentclass[11pt,a4paper,twoside]{article}
\usepackage{amssymb}
\usepackage{amsfonts}
\usepackage{geometry}
\usepackage{color}
\usepackage{mathtools}
\usepackage{amsmath,amsthm,amssymb,amsfonts}
\usepackage{CJK}
\usepackage{latexsym}
\usepackage{graphicx}
\usepackage{pgfpages}
\usepackage{mathrsfs}
\usepackage{ytableau}
\usepackage{enumerate}
\usepackage{comment}
\usepackage{cite}
\usepackage[title,titletoc]{appendix}
\usepackage[symbols,nogroupskip,sort=none]{glossaries-extra}
\usepackage[hyphens]{url}
\usepackage[bookmarks=false,hidelinks, colorlinks, citecolor=red,linkcolor=blue]{hyperref}
\usepackage[affil-it]{authblk}
\allowdisplaybreaks

\newtheorem{defn}{Definition}[section]

\newtheorem{lemma}[defn]{Lemma}
\newtheorem{proposition}[defn]{Proposition}
\newtheorem{corollary}[defn]{Corollary}
\newtheorem{remark}[defn]{Remark}

\begin{document}
	
		\title{Two Colored Diagrams for Central  Configurations of the Planar Five-vortex Problem}
	\author [1] {Xiang Yu}
	\author [2] {Shuqiang Zhu}
	\affil[1] {Center for Applied Mathematics and KL-AAGDM, Tianjin University, Tianjin,  300072, China}
	\affil[2] {School of Mathematics, Southwestern 	University of Finance and Economics, Chengdu 611130, China}
	\affil[ ]{xiang.zhiy@foxmail.com, yuxiang\_math@tju.edu.cn, zhusq@swufe.edu.cn}
	\date{}
	
	\maketitle

\begin{abstract}
	We apply the  singular sequence method to investigate   the finiteness problem for stationary configurations of the planar five-vortex problem. 
	The initial step of the singular sequence method involves identifying all two-colored diagrams. These diagrams represent potential scenarios where finiteness may fail. We determined all such diagrams for the planar five-vortex problem. 
\end{abstract}


\section{Introduction}

The \emph{planar $N$-vortex problem} which originated from Helmholtz's work in 1858 \cite{helmholtz1858integrale}, considers the motion of point vortices in a fluid plane. It was given a Hamiltonian formulation by Kirchhoff as follows:
\[
\Gamma_n \dot{\mathbf{r}}_n = J \frac{\partial H}{\partial \mathbf{r}_n} = J \sum_{1 \leq j \leq N, j \ne n} \Gamma_j \Gamma_n \frac{\mathbf{r}_j - \mathbf{r}_n}{|\mathbf{r}_j - \mathbf{r}_n|^2}, \quad n = 1, \ldots, N.
\]
Here, $J = \begin{bmatrix} 0 & 1 \\ -1 & 0 \end{bmatrix}$, 
$\mathbf{r}_n = (x_n, y_n) \in \mathbb{R}^2$, and $\Gamma_n$ $(n = 1, \ldots, N)$ are the positions and vortex strengths (or vorticities) of the vortices, and 
the Hamiltonian is $H = -\sum_{1 \leq j < k \leq N} \Gamma_j \Gamma_k \ln|\mathbf{r}_j - \mathbf{r}_k|$, where $|\cdot|$ denotes the Euclidean norm in $\mathbb{R}^2$. The $N$-vortex problem is a widely used model for providing finite-dimensional approximations to vorticity evolution in fluid dynamics, especially when the focus is on the trajectories of the vorticity centers rather than the internal structure of the vorticity distribution \cite{Newton2001}.

An interesting set of special solutions of the dynamical system are homographic solutions, where the relative shape of the configuration remains constant during the motion. An excellent review of these solutions can be found in \cite{aref2003vortex, Newton2001}. 
Following O'Neil, we refer to the corresponding configurations as \emph{stationary}. The only stationary configurations are equilibria, rigidly translating configurations (where the vortices move with a common velocity), relative equilibria (where the vortices rotate uniformly), and collapse configurations (where the vortices collide in finite time) \cite{o1987stationary}.

 The equations governing stationary configurations are similar to those describing central configurations in celestial mechanics. 
Albouy and Kaloshin introduced a novel method to study the finiteness of five-body central configurations in celestial mechanics \cite{Albouy2012Finiteness}. The first author successfully extended this approach to fluid mechanics. Using this new method, the first author established not only the finiteness of four-vortex relative equilibria for any four nonzero vorticities but also the finiteness of four-vortex collapse configurations for a fixed angular velocity. This represents the first result on the finiteness of collapse configurations for $N \geq 4$ \cite{hampton2009finiteness, yu2021Finiteness}.

In this paper, we focus on the finiteness of five-vortex relative equilibria and collapse configurations.   We apply the singular sequence method developed by the first author in \cite{yu2021Finiteness}. 	The initial step of the singular sequence method involves identifying all two-colored diagrams. These diagrams represent potential scenarios where finiteness may fail. We determined all such diagrams for the five-vortex problem.

The paper is structured as follows. In Sect. \ref{sec:basicnotations}, we introduce notations and definitions. In Sect. \ref{sec:pri},  we   briefly review the singular sequence method and the two-colored diagrams. In Sect. \ref{sec:moreproperty}, we identify constraints when some particular sub-diagrams appear. In Sect. \ref{sec:matrixrules}, we determine all possible two-colored diagrams for the planar five-vortex problem. 
In Sect. \ref{sec:diagram&constraints}, we sketch those diagrams.

\section{Basic notations}\label{sec:basicnotations}
\indent\par
We recall some basic notations on stationary configurations and direct readers to a more comprehensive introduction provided by O'Neil \cite{o1987stationary} and Yu \cite{yu2021Finiteness}.

We represent vortex positions $\mathbf{r}_n \in \mathbb{R}^2$ as complex numbers $z_n \in \mathbb{C}$. The equations of motion are  $\dot{z}_n = \textbf{i}V_n$, where
\begin{equation}\label{vectorfield}
	V_n= \sum_{1 \leq j \leq N, j \neq n} \frac{\Gamma_j z_{jn}}{r_{jn}^2}= \sum_{ j \neq n} \frac{\Gamma_j }{{\overline{z}_{jn}}}.
\end{equation}
Here, $z_{jn} = z_n - z_j$, $r_{jn} = |z_{jn}| = \sqrt{z_{jn}{\overline{z}_{jn}}}$, $\textbf{i} = \sqrt{-1}$, and the overbar denotes complex conjugation.

Let $\mathbb{C}^N= \{ z = (z_1,  \ldots, z_N):z_j \in \mathbb{C}, j = 1,  \ldots, N \}$ denote the space of configurations for $N$ point vortex. The collision set is defined as  $\Delta=\{ z \in \mathbb{C}^N:z_j=z_k ~~\emph{for some}~~ j\neq k  \} $. The space of collision-free configurations is given by $\mathbb{C}^N \backslash \Delta$.

\begin{defn}\label{def:LI}
	The following quantities and notations are defined:
	\begin{center}
		$\begin{array}{cc}
		\text{Total vorticity} & \Gamma =\sum_{j=1}^{N}\Gamma_j  \\
		\text{Total vortex angular momentum} & L =\sum_{1\leq j<k\leq N}\Gamma_j\Gamma_k.  
		\end{array}$
	\end{center}
	For $J=\{j_1, ..., j_n\}\subset \{1, ..., N\}$, we also define 
	\[ \Gamma_J=\Gamma_{j_1, ..., j_n}=\sum_{j\in J} \Gamma_j, \ L_J=L_{j_1, ..., j_n}=\sum_{j<k, j,k \in J} \Gamma_j \Gamma_k.\] 
\end{defn}

 A motion is called homographic if the relative shape remains constant. 
	Following  O'Neil \cite{o1987stationary}, we term  a corresponding  configuration as a  \emph{stationary configuration}.  Equivalently, 
\begin{defn} \label{def-1}
	A configuration $z \in \mathbb{C}^N \backslash \Delta$ is stationary if there exists a constant $\Lambda\in {\mathbb{C}}$ such that
	\begin{equation}\label{stationaryconfiguration}
		V_j-V_k=\Lambda(z_j-z_k), ~~~~~~~~~~ 1\leq j, k\leq N.
	\end{equation}
\end{defn}

 There are  only four kinds of homographic motions,  equilibria, translating with a common velocity, uniformly rotating, and   homographic motions that collapse in finite time.   
	Following \cite{o1987stationary, hampton2009finiteness, yu2021Finiteness},  we term 
	the  stationary
	configurations  corresponding to these four classes of  homographic motions as
	equilibria, rigidly translating configurations, relative equilibria  and collapse configurations. 
	Equivalently,

\begin{defn}\label{def-2}
	\begin{itemize}
		\item[i.] $z \in \mathbb{C}^N \backslash \Delta$ is an \emph{equilibrium} if $V_1=\cdots=V_N=0$.
		\item[ii.] $z \in \mathbb{C}^N \backslash \Delta$ is \emph{rigidly translating} if $V_1=\cdots=V_N=c$ for some $c\in \mathbb{C}\backslash\{0\}$.
		\item[iii.] $z \in \mathbb{C}^N \backslash \Delta$ is  a \emph{relative equilibrium} if there exist constants $\lambda\in \mathbb{R}\backslash\{0\},z_0\in \mathbb{C}$ such that $V_n=\lambda(z_n-z_0),~~~~~~~~~~ 1\leq n\leq N$.
		\item[iv.] $z \in \mathbb{C}^N \backslash \Delta$ is a \emph{collapse configuration} if there exist constants $\Lambda,z_0\in \mathbb{C}$ with $\emph{Im}(\Lambda)\neq0$ such that $V_n=\Lambda(z_n-z_0),~~~~~~~~~~ 1\leq n\leq N$.
	\end{itemize}
\end{defn}

\begin{defn}
	A configuration $z$ is equivalent to $z'$ if there exist $a, b \in \mathbb{C}$ with $b \neq 0$ such that $z'_n = b(z_n + a)$ for $1 \leq n \leq N$.

		A configuration is called translation-normalized if its translation freedom is removed, rotation-normalized if its rotation freedom is removed, and dilation-normalized if its dilation freedom is removed. 
		A configuration normalized in translation, rotation, and dilation is termed a \textbf{normalized configuration}. 
\end{defn}

	We count the stationary configurations according to the equivalence classes. Counting equivalence classes is the same as counting normalized configurations. Note that the removal of any of these three freedoms can be performed in various  ways.

\section{Singular sequences for central configurations and coloring rules}\label{sec:pri} 
\label{Preliminaries}

In this section, we briefly review   the basic elements of the Albouy-Kaloshin approach developed   by Yu  \cite{yu2021Finiteness} for the finiteness of relative equilibria and collapse configurations, including,  among others, the notation of central configurations, the extended system, the notation of singular sequences, the two-colored diagrams, and the rules for  the two-colored diagrams. For  a more comprehensive introduction, please refer to   \cite{yu2021Finiteness}.

\subsection{Central configurations of the planar N-vortex problem}
Recall Definition \ref{def-2}. Equations of relative equilibria and collapse configurations share the form:
\begin{equation}\label{stationaryconfiguration1}
	V_n=\Lambda(z_n-z_0),~~~~~~~~~~ 1\leq n\leq N,
\end{equation}
where $\Lambda\in \mathbb{R}\backslash\{0\}$  indicates relative equilibria and $\Lambda\in \mathbb{C}\backslash\mathbb{R}$  indicates  collapse configurations.
\begin{defn}\label{def:cc}
	Relative equilibria and collapse configurations are both called \textbf{central configurations}.
\end{defn}

The equations (\ref{stationaryconfiguration1}) read 
\begin{equation}\label{stationaryconfiguration2}
	\Lambda z_n= V_n,~~~~~~~~~~ 1\leq n\leq N,
\end{equation}
if the translation freedom is removed, i.e., we substitute $z_n$ with $z_n + z_0$ in equations (\ref{stationaryconfiguration2}). The solutions then satisfy:
\begin{equation}\label{center0}
	M=0, \ \Lambda I= L.
\end{equation}
To remove dilation freedom, we enforce $|\Lambda| = 1$.

Introduce a new set of variables $w_n$ and a ``conjugate"
relation:
\begin{equation}\label{stationaryconfiguration3}
	\Lambda z_n=\sum_{ j \neq n} \frac{\Gamma_j }{{w_{jn}}},\ \ 
	\overline{\Lambda} w_n=\sum_{ j \neq n} \frac{\Gamma_j }{{z_{jn}}},\ \ \  1\leq n\leq N,
\end{equation}
where $z_{jn}=z_{n}-z_{j}$ and $w_{jn}=w_{n}-w_{j}$.

The rotation symmetry of \eqref{stationaryconfiguration2} leads to   the invariance of
(\ref{stationaryconfiguration3}) under the map 
\[R_a: (z_1, ..., z_n, w_1, ..., w_n) \mapsto (az_1, ..., az_N, a^{-1} w_1, ..., a^{-1}w_N)\]
for any $a\in\mathbb{C}\backslash \{0\}$. 

Introduce  the variables $Z_{jk},W_{jk}\in \mathbb{C}$ $(1\leq j< k\leq N)$ such that
$Z_{jk}=1/w_{jk}, W_{jk}=1/z_{jk}$. For $1\leq k< j\leq N$ we set $Z_{jk}=-Z_{kj}, W_{jk}=-W_{kj}$. Then equations (\ref{stationaryconfiguration2}) together with the condition $z_{12}\in\mathbb{R}$ and $|\Lambda|=1$ are embedded into the following extended system
\begin{equation}\label{equ:complexcc}
	\begin{array}{cc}
		\Lambda z_n=\sum_{ j \neq n} \Gamma_j Z_{jn},&1\leq n\leq N, \\
		\overline{\Lambda} w_n=\Lambda^{-1} w_n=\sum_{ j \neq n} \Gamma_j W_{jn},& 1\leq n\leq N, \\
		Z_{jk} w_{jk}=1,&1\leq j< k\leq N, \\
		W_{jk} z_{jk}=1,&1\leq j< k\leq N, \\
		z_{jk}=z_k-z_j,~~~  w_{jk}=w_k-w_j,&1\leq j, k\leq N, \\
		Z_{jk}=-Z_{kj},~~~ W_{jk}=-W_{kj},&1\leq k< j\leq N, \\
		z_{12}=w_{12}.
	\end{array}
\end{equation}
This is a polynomial system in the  variables $\mathcal{Q}=(\mathcal{Z},\mathcal{W})\in\mathbb{C}^{2\mathfrak{N}}$, here
\begin{center}
	$\mathcal{Z}=(\mathcal{Z}_{1},\mathcal{Z}_{2},\ldots,\mathcal{Z}_{\mathfrak{N}})=(z_1,z_2,\ldots,z_N,Z_{12},Z_{13},\ldots,Z_{(N-1)N})$,
	$\mathcal{W}=(\mathcal{W}_{1},\mathcal{W}_{2},\ldots,\mathcal{W}_{\mathfrak{N}})=(w_1,w_2,\ldots,w_N,W_{12},W_{13},\ldots,W_{(N-1)N})$.
\end{center}
and $\mathfrak{N}=N(N+1)/2$.  

\begin{defn}\label{def:positivenormalizedcentralconfiguration}
 A \emph{complex normalized} central configuration of the planar
		$N$-vortex problem is a solution of (\ref{equ:complexcc}).  A \emph{real  normalized} central configuration of the planar
		$N$-vortex problem is a  complex normalized central configuration  satisfying $z_n={\overline{w}}_n$ for any $n=1, \ldots, N$.  
\end{defn}

 Note that a real  normalized central configuration of Definition \ref{def:positivenormalizedcentralconfiguration} is exactly  a central configuration of 
	Definition \ref{def:cc}.  	
We will use the name ``distance" for the $r_{jk}=\sqrt{z_{jk}{w_{jk}}}$. Strictly speaking, the distances $r_{jk}=\sqrt{z_{jk}{w_{jk}}}$ are now bi-valued. However, only the squared distances appear in the system, so we shall  understand $r^2_{jk}$ as $z_{jk}w_{jk}$ from now on.

\subsection{Singular sequences}
Let $\|\mathcal{Z}\|=\max_{j=1,2,\ldots,\mathfrak{N}}|\mathcal{Z}_{j}|$ be the modulus of the maximal component of
the vector $\mathcal{Z}\in \mathbb{C}^\mathfrak{N}$. Similarly, set  $\|\mathcal{W}\|=\max_{k=1,2,\ldots,\mathfrak{N}}|\mathcal{W}_{k}|$.

One important feature of System \eqref{equ:complexcc} is the symmetry: 
if $\mathcal Z, \mathcal W$ is a solution, so is $ a\mathcal Z, a^{-1} \mathcal W$ for any $a\in\mathbb{C}\backslash\{0\}$. Thus, we can replace 
the  normalization $z_{12}=w_{12}$  in System \eqref{equ:complexcc}  by 
$\|\mathcal{Z}\|=\|\mathcal{W}\|$.  From now on, we consider System \eqref{equ:complexcc} with this new normalization.

Consider
a sequence $\mathcal{Q}^{(n)}$, $n=1,2,\ldots$, of solutions  of (\ref{equ:complexcc}).  Take a sub-sequence such that the maximal component of $\mathcal{Z}^{(n)}$ is fixed, i.e., there is a $j\in \{1,2,\ldots,\mathcal{N}\}$ that is  independent 
of $n$ such that  $\|\mathcal{Z}^{(n)}\|=|\mathcal{Z}^{(n)}_{j}|$. 
Extract again in such a way that the sequence $\mathcal{Z}^{(n)}/\|\mathcal{Z}^{(n)}\|$ converges.
Extract again  in such a way that  the maximal component of $\mathcal{W}^{(n)}$ is fixed. Finally,  extract  in
such a way that the sequence $\mathcal{W}^{(n)}/\|\mathcal{W}^{(n)}\|$ converges.

\begin{defn}[Singular sequence]
	Consider a sequence of complex normalized central configurations with the property  that $\mathcal{Z}^{(n)}$ is unbounded. A
	sub-sequence extracted by the above process is called
	a \emph{singular sequence}.
\end{defn}

\begin{lemma}\label{Eliminationtheory}\cite{Albouy2012Finiteness}
	Let $\mathcal{X}$ be a closed algebraic subset of $\mathbb{C}^m$ and $f:\mathbb{C}^m\rightarrow \mathbb{C}$ be a
	polynomial. Either the image $F(\mathcal{X})\subset\mathbb{ C}$ is a finite set, or it is the complement
	of a finite set. In the second case one says that f is dominating.
\end{lemma}

\subsection{The  two-colored diagrams} \label{sec:rule}

For two sequences of non-zero numbers, $a, b$, we use $a\sim b$,  $a\prec b$, $ a\preceq b$,  and $ a \approx b$ to represent ``$a/b\rightarrow 1$'', ``$a/b\rightarrow 0$'', ``$a/b$ is bounded'' and ``$a\preceq b$, $a\succeq  b$'' respectively.  

Recall that  a  singular sequence satisfy the property   $\|\mathcal{Z}^{(n)}\|=\|\mathcal{W}^{(n)}\|\to \infty$.  Set $\|\mathcal{Z}^{(n)}\|=\|\mathcal{W}^{(n)}\|=1/\epsilon^2$. Then $\epsilon\rightarrow 0$.  Following Albouy-Kaloshin, \cite{Albouy2012Finiteness}, the \emph{two-colored diagram} was introduced in \cite{yu2021Finiteness} to  classify the singular sequences.  Given a singular sequence, the indices of the vertices  will be written down. 
The first color, called the $z$-color (red),   is used to mark the maximal order components of $\mathcal{Z}$. If $z_k\approx \epsilon^{-2}$,  draw a $z$-circle around the
vertex $\textbf{k}$; If   $Z_{jk}\approx \epsilon^{-2}$,  draw a $z$-stroke between vertices $\textbf{k}$ and $\textbf{j}$. 
They consist  the $z$-diagram. 
The second color, called the $w$-color (blue and dashed),   is used to mark the maximal order components of $\mathcal W$ in similar manner. Then we also have the  $w$-diagram. The two-colored diagram is the combination of the  $z$-diagram and the $w$-diagram,  see Figure \ref{fig:edges}.

If there
is either a $z$-stroke, or a $w$-stroke, or both between vertex $\textbf{k}$ and vertex $\textbf{l}$, we say that there is an edge between them.  There are three types of edges,
$z$-edges, $w$-edges and $zw$-edges, see Figure \ref{fig:edges}. 

\begin{figure}[h!]
	\centering
	\includegraphics[width=0.6\textwidth]{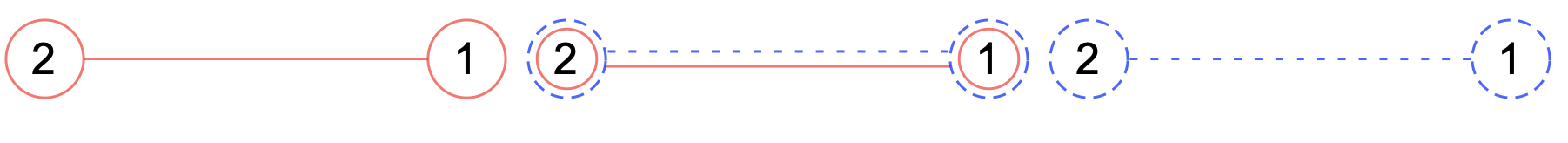} 
	\caption{On the left, vertices \textbf{1,2} are $z$-circled, and a $z$-edge is between them; In the middle, vertices \textbf{1,2} are $z$- and $w$-circled, and a $zw$-edge is between them; On the right,  vertices \textbf{1,2} are $w$-circled, and a $w$-edge is between them; 
	}
	\label{fig:edges}
\end{figure}

The following concepts were introduced to characterize some features of singular sequences.    In the $z$-diagram,   vertices  $\textbf{k}$ and $\textbf{l}$
are  called \emph{$z$-close},  if $z_{kl} \prec\epsilon^{-2}$; 
a $z$-stroke between  vertices   $\textbf{k}$ and  $\textbf{l}$ is  called  	a   \emph{maximal $z$-stroke} if $z_{kl} \approx  \epsilon^{-2}$; 
a subset of vertices are called \emph{an isolated component of the $z$-diagram} if there is no $z$-stroke  between a vertex of this subset and  a vertex of its complement. 
These concepts also apply to the $w$-diagram.

\begin{proposition}[Estimate]\label{Estimate1}\cite{yu2021Finiteness}
	For any $(k,l)$, $1\leq k<l\leq N$, we have $\epsilon^2\preceq z_{kl}\preceq \epsilon^{-2}$, $\epsilon^2\preceq w_{kl}\preceq \epsilon^{-2}$ and $\epsilon^2\preceq r_{kl}\preceq \epsilon^{-2}$.
	
	There is a $z$-stroke between $\textbf{k}$ and $\textbf{l}$ if and only if $w_{kl}\approx \epsilon^{2}$, then $ r_{kl}\preceq 1$.
	
	There is a maximal $z$-stroke between $\textbf{k}$ and $\textbf{l}$ if and only if $z_{kl}\approx \epsilon^{-2}, w_{kl}\approx \epsilon^{2}$, then $ r_{kl}\approx1$.
	
	There is a $z$-edge between $\textbf{k}$ and $\textbf{l}$ if and only if $z_{kl}\succ \epsilon^{2},w_{kl}\approx \epsilon^{2}$, then $\epsilon^{2}\prec r_{kl}\preceq 1 $.
	
	There is a maximal $z$-edge between $\textbf{k}$ and $\textbf{l}$ if and only if $z_{kl}\approx \epsilon^{-2},w_{kl}\approx \epsilon^{2}$, then $ r_{kl}\approx 1$.
	
	There is a $zw$-edge between $\textbf{k}$ and $\textbf{l}$ if and only if $z_{kl},w_{kl}\approx \epsilon^{2}$, this can be  characterized as $ r_{kl}\approx \epsilon^{2}$.
\end{proposition}

\begin{remark}
	By the estimates above, the strokes in a $zw$-edge are not maximal. A maximal $z$-stroke is exactly a maximal $z$-edge.
\end{remark}

The following rules for the two-colored diagrams  are valid if  ``$z$'' and ``$w$''  were switched.

\begin{description}
	\item[{Rule I}]
	There is something at each end of any $z$-stroke: another $z$-stroke
	or/and a $z$-circle drawn around the name of the vertex. A $z$-circle cannot be isolated; there must be a $z$-stroke emanating from it. There is at least one
	$z$-stroke in the $z$-diagram.
	\item[{Rule II}] If vertices $\textbf{k}$ and $\textbf{l}$
	are  $z$-close, they are both $z$-circled or both not
	$z$-circled.
	\item[{Rule III}]  The moment of vorticity of a set of vertices forming an isolated component of the $z$-diagram is $z$-close to the origin.
	\item[{Rule IV}]  Consider the $z$-diagram or an isolated component of it. If there
	is a $z$-circled vertex, there is another one.   If the $z$-circled vertices are all
	$z$-close together,  the total vorticity of these $z$-circled  vertices is zero.
	\item[{Rule V}]  There is at least one $z$-circle at certain end of any maximal $z$-stroke. As a result,
	if an isolated component of the $z$-diagram has no $z$-circled vertex,
	then it has no maximal $z$-stroke.
	\item[{Rule VI}]
	If there are two consecutive $z$-stroke, there is a third $z$-stroke closing the triangle.
\end{description}




\section{Constraints  when some sub-diagrams appear} \label{sec:moreproperty}

\indent\par
We collect some useful results in this section. 
	Recall that  $\Gamma_J, \Gamma_{j_1, ..., j_n}$, $L_J$ and $L_{j_1, ..., j_n}$  are defined in  Definition \ref{def:LI}.

\begin{proposition}\label{Prp:sumT12} \cite{yu2021Finiteness}
	Suppose that a diagram has two $z$-circled vertices (say $\textbf{1}$ and $\textbf{2}$) which are also $z$-close,   if none of all the other vertices is $z$-close with them,  then $\Gamma_1+\Gamma_2\neq 0$ and $\overline{\Lambda}z_{12}w_{12}\sim \frac{1}{\Gamma_1+\Gamma_2}$. In particular, vertices $\textbf{1}$ and $\textbf{2}$ cannot form a $z$-stroke.
\end{proposition}

\begin{corollary}\label{Cor:sumT12} 
	Suppose that a diagram has  two $z$-circled vertices (say $\textbf{1}$ and $\textbf{2}$) which also form a $z$-stroke. If  none of all the other vertices is $z$-close with them, then $z_{12}\approx \epsilon^{-2}$,  $\Gamma_1+\Gamma_2\neq 0$, and $w_1, w_2\preceq \epsilon ^2$.  
\end{corollary}

\begin{proof}
	If they are $z$-close, by Proposition \ref{Prp:sumT12}, they cannot form a $z$-stroke, which is a contradiction. Note that    Rule III implies  
	\[\epsilon^{-2}   \succ  \Gamma_1 z_1 +\Gamma_2 z_2 = (\Gamma_1+\Gamma_2) z_1 +\Gamma_2 (z_2-z_1).   \]
	We obtain $\Gamma_1+\Gamma_2\neq 0.$
	
	Note that $z_{1j}, z_{2j}\approx  \epsilon^{-2}, j\ge3$. Then 
	\[\bar{\Lambda}\sum_{j\ge 3} \Gamma_j w_j=\sum_{j\ge 3}\frac{\Gamma_1\Gamma_j}{z_{1j}}+\sum_{j\ge3}\frac{\Gamma_2\Gamma_j}{z_{2j}}\preceq \epsilon^{2}.\]
	By the equation $\sum_{j}\Gamma_j w_j=0$, we have 
	\[ \epsilon^{2}\succeq \Gamma_1 w_1+\Gamma_2w_2=(\Gamma_1+\Gamma_2)w_1+ \Gamma_2w_{21}.  \]
	Since $w_{21}\approx \epsilon^2$, we have $w_1, w_2\preceq \epsilon ^2$.  
\end{proof}

\begin{proposition}\label{Prp:LI} \cite{yu2021Finiteness}
	Suppose that  a fully $z$-stroked sub-diagram with  vertices $\{1,..., k\}, (k\ge 3)$ exists in isolation  in a diagram, and none of its vertices  is $z$-circled, then
	\begin{equation}\notag
		L_{1...k}= \sum_{i,j \in \{1, ..., k\}, i\ne j}\Gamma_i\Gamma_j=0.
	\end{equation}
	
\end{proposition}

\begin{corollary}\label{Cor:L&Gamma} 
	Suppose a fully \(z\)-stroked sub-diagram with vertices \(K = \{1, \ldots, k\}\), \(k \ge 3\), exists in isolation in a diagram, and none of its vertices is \(z\)-circled. 
	\begin{enumerate}
		\item If there is an isolated component \(I\) of the \(w\)-diagram such that \(K \subset I\), then the \(w\)-circled vertices in \(I\) cannot be exactly \(\{1, \ldots, k\}\). 
		\item Consider  any subset of \(K\) with cardinality \((k-1)\) , say \(K_1 = \{2, \ldots, k\}\). If there is an isolated component \(I\) of the \(w\)-diagram with \(K_1 \subset I\), then the \(w\)-circled vertices in \(I\) cannot be exactly \(K_1\). 
		\item If there is a vertex outside of \(K\), say \(k+1\), such that \(\{k+1\} \cup K\) forms an isolated component of the \(w\)-diagram and these \(k+1\) vertices are fully \(w\)-stroked, then there is at least one \(w\)-circle among them.  
		\item If there are several isolated components \(\{I_j, j = 1, \ldots, s\}\) of the \(w\)-diagram with \(K \subset \cup_{j=1}^s I_j\), then the \(w\)-circled vertices in \(\cup_{j=1}^s I_j\) cannot be exactly \(\{1, \ldots, k\}\). 
	\end{enumerate}
\end{corollary}

\begin{proof}
	First, we have \(L_{K} = 0\) by Proposition \ref{Prp:LI}, and the vertices of \(K\) are all \(w\)-close by  the estimate of Proposition \ref{Estimate1}.
	
	For part (1), if the \(w\)-circled vertices in \(I\) are exactly \(\{1, \ldots, k\}\), then by Rule IV, we have \(\sum_{i \in K} \Gamma_i = 0\). This leads to a contradiction because:
	\[
	\left(\sum_{i \in K} \Gamma_i\right)^2 = \sum_{i \in K} \Gamma_i^2 + 2L_{K}.
	\]
	The proof of part (4) is similar.
	
	For part (2), if the \(w\)-circled vertices in \(I\) are exactly \(K_1 = \{2, \ldots, k\}\), then by Rule IV, we have \(\sum_{i=2}^k \Gamma_i = 0\). Therefore:
	\[
	L_{K_1} = L_{K} - \Gamma_{1}\left(\sum_{i=2}^k \Gamma_i\right) = 0,
	\]
	which again leads to a contradiction since \(\sum_{i \in K_1} \Gamma_i = 0\).
	
	For part (3), if the component \(\{k+1\} \cup K\) is fully \(w\)-stroked but has no \(w\)-circle, then Rule IV implies that \(L_K = 0\) and \(L_K + \Gamma_{k+1} \sum_{i \in K} \Gamma_i = 0\). This leads to \(\sum_{i \in K} \Gamma_i = 0\), which is a contradiction.
\end{proof}

\begin{proposition}\label{Prp:isolate-z_12}
	Suppose that a diagram has an isolated \(z\)-stroke in the \(z\)-diagram, and its two ends are \(z\)-circled. Let \(\textbf{1}\) and \(\textbf{2}\) be the ends of this \(z\)-stroke. Suppose there is no other \(z\)-circle in the diagram. Then \(\Gamma_1 + \Gamma_2 \ne 0\), and \(z_{12}\) is maximal. The diagram forces \(\Lambda = \pm 1\) or \(\pm \textbf{i}\).
	Furthermore,
	\begin{itemize}
		\item If $\Lambda = \pm 1$, we have $\sum_{j=3}^N \Gamma_j=0$;
		\item 	If $\Lambda = \pm \textbf{i}$, we have $L=0$ and $\Gamma_1\Gamma_2=L_{3...N}$.
	\end{itemize}
\end{proposition}

\begin{proof}
	
	The facts  that	$\Gamma_1+\Gamma_2\neq 0$ and $z_{12}$ is maximal  follow from Corollary \ref{Cor:sumT12}. 
	Without loss of generality, assume $z_1\sim -\Gamma_2 a\epsilon^{-2}$ and $z_2\sim \Gamma_1 a \epsilon^{-2}$, then $$z_{12}\sim (\Gamma_1 +\Gamma_2) a \epsilon^{-2},\  \frac{1}{z_{2}}-\frac{1}{z_{1}}\sim (\frac{1}{\Gamma_1}+\frac{1}{\Gamma_2})\frac{\epsilon^2}{a}. $$
	
	The  System \eqref{equ:complexcc} yields
	\begin{equation}
		\begin{array}{c}
			\label{equ:iso-z12}\overline{\Lambda} w_{12}=
			(\Gamma_1+\Gamma_2)W_{12}+ \sum_{j=3}^N \Gamma_j (\frac{1}{z_{j2}}-\frac{1}{z_{j1}}) \cr 
			\Lambda z_{2} \sim \Gamma_1 Z_{12}.
		\end{array}{}
	\end{equation}
	
	The second equation of \eqref{equ:iso-z12} implies  $w_{12} \sim \frac{\epsilon^2}{ a\Lambda}$.  Note that
	$\frac{1}{z_{j2}}-\frac{1}{z_{j1}}\sim \frac{1}{z_{2}}-\frac{1}{z_{1}}$
	for all $j >2$ and that  $W_{12}=\frac{1}{z_{12}}$. The first equation of
	\eqref{equ:iso-z12} implies 
	\begin{equation}\label{equ:iso-z12-1}
		{\overline{\Lambda}}/{\Lambda}=1+ \sum_{j=3}^N \Gamma_j  (\frac{1}{\Gamma_1}+\frac{1}{\Gamma_2}).
	\end{equation}
	It follows that $\Lambda = \pm 1$ or $ \pm \textbf{i}$.
	
	If $\Lambda = \pm 1$, we have
	\[ 0= \sum_{j=3}^N \Gamma_j  (\frac{1}{\Gamma_1}+\frac{1}{\Gamma_2}),\  \Rightarrow\  \sum_{j=3}^N \Gamma_j =0. \]
	
	If $\Lambda = \pm \textbf{i}$, we obtain 
	\[ -2= \sum_{j=3}^N \Gamma_j  (\frac{1}{\Gamma_1}+\frac{1}{\Gamma_2}),\  L=0, \ \Rightarrow L=0, \ \Gamma_1 \Gamma_2=L_{3...N}.  \]
\end{proof}

Similarly, we have the following result.

\begin{proposition}\label{Prp:isolate-z_123}
	Suppose that a diagram has an isolated triangle of \(z\)-strokes in the \(z\)-diagram, where two of the vertices of the triangle are \(z\)-circled. Let \(\textbf{2}\) and \(\textbf{3}\) be the two \(z\)-circled vertices and \(\textbf{1}\) the other vertex. Suppose there is no other \(z\)-circle in the diagram. Then \(\Gamma_2 + \Gamma_3 \ne 0\), and \(z_{23}\) is maximal. The diagram forces \(\Lambda = \pm 1\) or \(\pm \textbf{i}\).  Furthermore,
	\begin{itemize}
		\item If $\Lambda = \pm 1$, we have $\sum_{j=4}^N \Gamma_j=0$;
		\item 	If $\Lambda = \pm \textbf{i}$, we have $L=0$ and $L_{123}=L_{4...N} + \Gamma_1( \sum _{j=4}^N \Gamma_j)$.
	\end{itemize}
\end{proposition}

\begin{proof}
	The facts that 	$\Gamma_2+\Gamma_3\ne 0$ and $z_{23}$ is maximal  follow from Corollary \ref{Cor:sumT12}.  Note that
	\[  \Gamma_2 z_2 +\Gamma_3 z_3 \prec \epsilon^{-2}, \ \Lambda  z_1 \sim \Gamma_2 Z_{21} +\Gamma_3 Z_{31} \prec \epsilon^{-2}.  \]
	Without loss of generality, assume
	\[ z_2\sim -\Gamma_3 a\epsilon^{-2}, \ z_3\sim \Gamma_2 a \epsilon^{-2}, \ Z_{21} \sim -\Gamma_3 b\epsilon^{-2}, \ Z_{31}\sim \Gamma_2 b \epsilon^{-2}.  \]
	Then
	\begin{eqnarray*}
		z_{23}\sim (\Gamma_2 +\Gamma_3) a \epsilon^{-2},\  \frac{1}{z_{3}}-\frac{1}{z_{2}}\sim (\frac{1}{\Gamma_2}+\frac{1}{\Gamma_3})\frac{\epsilon^2}{a}, \cr
		Z_{23}=\frac{1}{w_{23}}= \frac{1}{1/Z_{21} +1/Z_{13}} \sim -b \frac{\Gamma_2\Gamma_3}{(\Gamma_2+\Gamma_3)}  \epsilon^{-2}.
	\end{eqnarray*}
	Then similar to the above case, we have
	\begin{eqnarray*}
		\overline{\Lambda} w_{23}\sim
		(\Gamma_2+\Gamma_3)W_{23}+ \sum_{j\ne 2, 3}^N \Gamma_j (\frac{1}{z_{3}}-\frac{1}{z_{2}}) \cr 
		\Lambda z_{23} \sim  (\Gamma_2+\Gamma_3) Z_{23} +\Gamma_1 (Z_{13}-Z_{12}).
	\end{eqnarray*}
	
	Short computation reduces the two equations to
	\begin{eqnarray*}
		-\overline{\Lambda} \frac{a}{b}= \frac{\Gamma_2\Gamma_3}{\Gamma_2+\Gamma_3} (1+\sum_{j\ne 2, 3} \Gamma_j \frac{\Gamma_2+\Gamma_3} {\Gamma_2\Gamma_3}), \cr
		-\Lambda \frac{a}{b}= \frac{L_{123}}{\Gamma_2+\Gamma_3}.
	\end{eqnarray*}
	Then we obtain
	\[   \frac{\overline{\Lambda}}{\Lambda}  L_{123}= \Gamma_2\Gamma_3+ (\Gamma_2+\Gamma_3) \sum_{j\ne 2, 3} \Gamma_j.   \]
	It follows that $\Lambda = \pm 1$ or $ \pm \textbf{i}$.
	
	If $\Lambda = \pm 1$, we have  $\sum_{j=4}^N \Gamma_j=0$.
	If $\Lambda = \pm \textbf{i}$, we have $L=0$ and
	$$L=0, \  -L_{123}= \Gamma_2\Gamma_3+ (\Gamma_2+\Gamma_3) \sum_{j\ne 2, 3} \Gamma_j,$$
	which is equivalent to
	$L_{123}=L_{4...N} + \Gamma_1( \sum _{j=4}^N \Gamma_j), \ L=0. $
\end{proof}

\begin{proposition}\label{Prp:triangle} 
	Assume there is a triangle with vertices \(\textbf{1}, \textbf{2}, \textbf{3}\) that is fully \(z\)- and \(w\)-stroked, and fully \(z\)- and \(w\)-circled. Moreover, assume that  the triangle is isolated in the \(z\)-diagram. Then there must exist  some \(k > 3\) such that \(z_{k1} \preceq 1\).
\end{proposition}

\begin{proof}
	By Proposition \ref{Estimate1} and Rule IV, we have  \[z_1\sim z_2\sim z_3, \ w_1\sim w_2\sim w_3, \ \Gamma_{1}+\Gamma_{2}+\Gamma_{3}=0.\]
	Suppose that it holds  $z_{k1}\succ 1$ for all $k>3$. Then  $\frac{1}{z_{kj}}- \frac{1}{z_{k1}}=\frac{z_{1j}}{z_{kj}z_{k1}} \prec \epsilon^2$ for all $k>3, 1\le j\le 3$, and so
	\[\bar{\Lambda}\sum_{j=1}^{3}\Gamma_j w_j=\sum_{k\ge 4} \sum_{j=1}^{3}\frac{\Gamma_k\Gamma_j}{z_{kj}}=\sum_{k\ge 4} \sum_{j=1}^{3}\Gamma_k\Gamma_j  (\frac{1}{z_{k1}}+\frac{1}{z_{kj}}- \frac{1}{z_{k1}} ) \prec \epsilon^{2}.\]

	By the fact that  $w_{12}, w_{13}, w_{23}\approx \epsilon^{2}$, the equations 
	\[\sum_{j=1}^{3}\Gamma_j w_j=\Gamma_{2}w_{12}+\Gamma_{3}w_{13}=\Gamma_{1}w_{21}+\Gamma_{3}w_{23}=\Gamma_{1}w_{31}+\Gamma_{2}w_{32} \prec \epsilon^2\]
	imply that 
	\begin{equation}\label{equ:triangle1}
		\frac{w_{12}}{\Gamma_{3}}\sim\frac{w_{23}}{\Gamma_{1}}\sim\frac{w_{31}}{\Gamma_{2}}\approx \epsilon^{2}.
	\end{equation}
	By the isolation of this triangle in the $z$-diagram, it holds that 
	\begin{equation}\label{equ:triangle2}
		\Lambda z_{1}\sim \frac{\Gamma_2}{w_{21}}+\frac{\Gamma_3}{w_{31}}, ~\Lambda z_{2}\sim \frac{\Gamma_1}{w_{12}}+\frac{\Gamma_3}{w_{32}},~\Lambda z_{3}\sim \frac{\Gamma_1}{w_{13}}+\frac{\Gamma_2}{w_{23}}.
	\end{equation}
	Since $z_1\sim z_2\sim z_3$, the equations \eqref{equ:triangle1} and \eqref{equ:triangle2} lead to 
	\[\frac{\Gamma_1}{\Gamma_{2}}-\frac{\Gamma_2}{\Gamma_{1}}=\frac{\Gamma_2}{\Gamma_{3}}-\frac{\Gamma_3}{\Gamma_{2}}=\frac{\Gamma_3}{\Gamma_{1}}-\frac{\Gamma_1}{\Gamma_{3}}.\] 
	This  contradicts with $\Gamma_{1}+\Gamma_{2}+\Gamma_{3}=0$. 
\end{proof}

Similarly, we have the following result.
\begin{proposition}\label{Prp:triangle2} 
	Suppose that a diagram has an isolated triangle of \(z\)-strokes in the \(z\)-diagram, where all three vertices, say \(\textbf{1}, \textbf{2}, \textbf{3}\), are \(z\)-circled. If \(z_1 \sim z_2 \sim z_3\), then there exists some \(k > 3\) such that \(z_{k1} \prec \epsilon^{-2}\). 
\end{proposition}

\begin{proof}
	Suppose that 	$z_1\sim z_2\sim z_3\approx \epsilon^{-2}$. By Proposition \ref{Estimate1} and Rule IV, we have $\Gamma_{1}+\Gamma_{2}+\Gamma_{3}=0$.
	Suppose that it holds  $z_{k1} \approx \epsilon^{-2}$ for all $k>3$. Then  $\frac{1}{z_{kj}}- \frac{1}{z_{k1}}=\frac{z_{1j}}{z_{kj}z_{k1}} \prec \epsilon^2$ for all $k>3, 1\le j\le 3$. Similar to the argument of the above result, we have  $\bar{\Lambda}\sum_{j=1}^{3}\Gamma_j w_j \prec \epsilon^{2}$, 
	\[  \frac{w_{12}}{\Gamma_{3}}\sim\frac{w_{23}}{\Gamma_{1}}\sim\frac{w_{31}}{\Gamma_{2}}\approx \epsilon^{2},\]
	\begin{equation}\notag
		\Lambda z_{1}\sim \frac{\Gamma_2}{w_{21}}+\frac{\Gamma_3}{w_{31}}, ~\Lambda z_{2}\sim \frac{\Gamma_1}{w_{12}}+\frac{\Gamma_3}{w_{32}},~\Lambda z_{3}\sim \frac{\Gamma_1}{w_{13}}+\frac{\Gamma_2}{w_{23}}, 
	\end{equation}
	and 
	$\frac{\Gamma_1}{\Gamma_{2}}-\frac{\Gamma_2}{\Gamma_{1}}=\frac{\Gamma_2}{\Gamma_{3}}-\frac{\Gamma_3}{\Gamma_{2}}=\frac{\Gamma_3}{\Gamma_{1}}-\frac{\Gamma_1}{\Gamma_{3}}.$
	This  contradicts with $\Gamma_{1}+\Gamma_{2}+\Gamma_{3}=0$. 
\end{proof}

\begin{proposition}\label{Prp:dumbbell} 
	Assume that vertices \textbf{1} and \textbf{2} are both $z$- and $w$-circled and connected by a $zw$-edge, and the sub-diagram formed by the two vertices  is isolated in the $z$-diagram. Assume that 
	vertices \textbf{3} and \textbf{4} are also both $z$- and $w$-circled and connected by a $zw$-edge,  and  is isolated in the $z$-diagram. 
	Then, there must exist some $k>4$ such that at least one among  $z_{k1}, w_{k1}, z_{k3}, w_{k3}$ is bounded (i.e., $\preceq 1$). 
\end{proposition}

\begin{proof}
	By Proposition \ref{Estimate1} and Rule IV, we have  \[z_1\sim z_2,  z_3\sim z_4, \ w_1\sim w_2, w_3 \sim w_4, \ \Gamma_{1}+\Gamma_{2}=0, \Gamma_3+\Gamma_{4}=0.\]
	Suppose that it holds that $w_{k1}\succ 1$ for all $k>4$. Then  $\frac{1}{w_{k2}}- \frac{1}{w_{k1}}=\frac{w_{12}}{w_{k2}w_{k1}} \prec \epsilon^2$ for all $k>4$.  
	Note that $z_{12}\approx \epsilon^2$, and 
	\[\Lambda z_{12}= (\Gamma_1+\Gamma_2)Z_{12}+ \Gamma_3( \frac{w_{21}}{w_{32}w_{31}} -\frac{w_{21}}{w_{42}w_{41}}) + \sum_{k>4}\Gamma_k(\frac{1}{w_{k2}}- \frac{1}{w_{k1}} ).     \]
	We conclude that  
	\[\frac{w_{21}}{w_{31}w_{32}}\approx  \frac{w_{21}}{w_{41}w_{42}}\succeq \epsilon^{2}\Rightarrow w_{31}\preceq 1\Rightarrow w_1\sim w_2\sim w_3\sim w_4.\]
	Similarly, we have\[z_1\sim z_2\sim z_3\sim z_4.\]

	Note that 
	\[\overline{\Lambda}\sum_{j=1}^{4}\Gamma_j w_j=\sum_{k> 4} \sum_{j=1}^{4}\frac{\Gamma_k\Gamma_j}{z_{kj}}=\sum_{k> 4} \Gamma_k \Gamma_1 (\frac{1}{z_{k1}}- \frac{1}{z_{k2}})  + \sum_{k> 4} \Gamma_k \Gamma_3 (\frac{1}{z_{k3}}- \frac{1}{z_{k4}})  \prec \epsilon^{2}. \]
	Then the equation $\Gamma_2 w_{12}+\Gamma_4 w_{34}=\sum_{j=1}^{4}\Gamma_j w_j$ leads to
	\[\Gamma_2 w_{12}\sim -\Gamma_4 w_{34}, \ \text{or} \ \Gamma_2 Z_{34}\sim -\Gamma_4 Z_{12} .\] 
	On the other hand, the isolation of the two segments implies 
	\[\Lambda z_2\sim \Gamma_1 Z_{12}, \ \Lambda z_4\sim \Gamma_3 Z_{34}, \Rightarrow \Gamma_1 Z_{12}\sim \Gamma_3 Z_{34}.\]
	As a result, we have
	\[\Gamma_1\Gamma_2=-\Gamma_3\Gamma_4, \ \text{or} \ \Gamma_1^2+\Gamma_3^2=0,\] 
	which is a contradiction.
	
\end{proof}

\begin{proposition}\label{Prp:quadrilateral} 
	Assume that there is a quadrilateral  with  vertices \textbf{1,  2, 3, 4}, that is fully $z$- and $w$-stroked, and fully  $w$-circled.   Moreover, the quadrilateral is isolated  in the $w$-diagram. 
	Then, there must exist some  $k>4$ such that $w_{k1} \preceq 1$. 
\end{proposition}

\begin{proof}
	We establish the result by contradiction. 
	Rule IV implies that   $\sum_{j=1}^{4}\Gamma_j=0.$ 	Suppose that it holds that $w_{k1}\succ 1$ for all $k>4$. Then $\frac{1}{w_{kj}}- \frac{1}{w_{k1}}\prec  \epsilon^{2}$ for $k>4, j\le4$, so
	\[\Lambda\sum_{j=1}^{4}\Gamma_j z_j=\sum_{k>4}\Gamma_k \sum_{j=1}^{4}\frac{\Gamma_j}{w_{kj}}=\sum_{k>4}\Gamma_k  \sum_{j=1}^{4} \Gamma_j  (\frac{1}{w_{k1}}+\frac{1}{w_{kj}}- \frac{1}{w_{k1}} )   \prec \epsilon^{2}.\]
	Then  \[  \epsilon^2 \succ \sum_{j=1}^{4}\Gamma_j z_j=\sum_{j=2, 3, 4}\Gamma_j z_{1j}\Rightarrow \Gamma_2 z_{12}\sim -\Gamma_3 z_{13}-\Gamma_4 z_{14}\approx \epsilon^{2}.  \]
	Set $z_{13}\sim a \epsilon^2, z_{14}\sim b \epsilon^2$, where $a\ne b$ are some nonzero constants. Then 
	\begin{align*}
		&z_{12}\sim -\frac{\Gamma_3 a+\Gamma_4 b}{\Gamma_2} \epsilon^2,    & z_{23}\sim \frac{a (\Gamma_2+\Gamma_3)+\Gamma_4 b}{\Gamma_2} \epsilon^2,\\
		&z_{24}\sim \frac{\Gamma_3 a+b (\Gamma_2+\Gamma_4)}{\Gamma_2} \epsilon^2,    &z_{34}\sim (b-a) \epsilon^2. 
	\end{align*}
	Since $w_1\sim w_2\sim w_3\sim w_4$, we set $\bar{\Lambda}w_k\sim \frac{1}{c\epsilon^2}, k=1, 2, 3, 4$.
	Substituting those into the system
	\[\bar{\Lambda}w_k\sim \sum_{j\ne k, j=1}^{4}\frac{\Gamma_j}{z_{jk}}, \ k=1, 2,3,4,\]
	which is from the isolation of the quadrilateral in $w$-diagram. 
	We obtain four homogeneous  polynomials of the three variables $a, b, c$. Thus, we set $c=1$, and obtain the following four polynomials of the five variables $a,b,  \Gamma_1, \Gamma_{3}, \Gamma_{4}$, 
	\begin{align*}
		&a^2 (-\Gamma_3) (b+\Gamma_4)+a b \left(-\Gamma_4 (b-2 \Gamma_3)+\Gamma_1^2+2 \Gamma_1 (\Gamma_3+\Gamma_4)\right)-b^2 \Gamma_3 \Gamma_4=0,\\
		&a^3 \Gamma_3^2 (\Gamma_1+\Gamma_4)+a^2 \Gamma_3 \left(-b \left(\Gamma_1^2+\Gamma_1 \Gamma_3+\Gamma_4 (2 \Gamma_3-\Gamma_4)\right)-(\Gamma_1-\Gamma_3+\Gamma_4) (\Gamma_1+\Gamma_3+\Gamma_4)^2\right)\\
		&-a b \left(\Gamma_1^2 \Gamma_4 (b-4 \Gamma_3)+\Gamma_1 \left(\Gamma_4^2 (b+2 \Gamma_4)+2 \Gamma_3^3-2 \Gamma_3^2 \Gamma_4-2 \Gamma_3 \Gamma_4^2\right)-\Gamma_3^2 \Gamma_4 (b+2 \Gamma_4)\right)\\
		&-a b \left(2 b \Gamma_3 \Gamma_4^2-\Gamma_1^4-2 \Gamma_1^3 (\Gamma_3+\Gamma_4)+\Gamma_3^4+\Gamma_4^4\right)\\
		&+b^2 \Gamma_4 \left(\Gamma_1 \left(\Gamma_4 (b+\Gamma_4)-3 \Gamma_3^2-2 \Gamma_3 \Gamma_4\right)+\Gamma_3 \Gamma_4 (b+\Gamma_4)-\Gamma_1^3-\Gamma_1^2 (3 \Gamma_3+\Gamma_4)-\Gamma_3^3-\Gamma_3^2 \Gamma_4+\Gamma_4^3\right)=0,\\
		&a^3 (\Gamma_1+\Gamma_4)+a^2 (\Gamma_3 (2 \Gamma_1+\Gamma_3+2 \Gamma_4)-b (\Gamma_1+2 \Gamma_4))+a b \left(b \Gamma_4-2 \Gamma_1 \Gamma_3-\Gamma_3^2-2 \Gamma_3 \Gamma_4\right)-b^2 \Gamma_1 \Gamma_4=0,\\
		&a^2 \Gamma_3 (b-\Gamma_1)-a b (b (\Gamma_1+2 \Gamma_3)+\Gamma_4 (2 \Gamma_1+2 \Gamma_3+\Gamma_4))+b^2 (b (\Gamma_1+\Gamma_3)+\Gamma_4 (2 \Gamma_1+2 \Gamma_3+\Gamma_4))=0. 
	\end{align*}
	Tedious but standard computation, such as calculating the Gr\"{o}bner basis,   yields 
	\[  b^5 (\Gamma_1+\Gamma_3+\Gamma_4) \left(\Gamma_1^2+\Gamma_1 \Gamma_3+\Gamma_1 \Gamma_4+\Gamma_3^2+\Gamma_3 \Gamma_4+\Gamma_4^2\right)=0. \]
	It is a contradiction since $b\ne 0, \Gamma_1+\Gamma_3+\Gamma_4=-\Gamma_2\ne 0$ and 
	\[\Gamma_1^2+\Gamma_1 \Gamma_3+\Gamma_1 \Gamma_4+\Gamma_3^2+\Gamma_3 \Gamma_4+\Gamma_4^2=\frac{1}{2}( \Gamma_1^2+\Gamma_2^2+\Gamma_3^2+\Gamma_4^2)\ne 0.\]
\end{proof}

\section{The thirty-one possible  diagrams for the planar five-vortex central configurations} \label{sec:matrixrules}

In this section,  we derive  all possible diagrams for the planar five-vortex central configurations. 
 As in the case of the 5-body  problem \cite{Albouy2012Finiteness} in celestial mechanics,  we divide possible diagrams into groups according to the maximal number of strokes from a bicolored vertex. During the analysis of all the possibilities we rule some of them out immediately. The  ones we cannot exclude without further work   are incorporated into the list of   thirty-one  diagrams,  shown in  Figure \ref{fig:list1} and \ref{fig:list2} of Section \ref{sec:diagram&constraints}.

We call a \emph{bicolored vertex} of the diagram a vertex which connects at least
a stroke of $z$-color with at least a stroke of $w$-color.  The number of strokes from
a bicolored vertex is at least 2 and at most 8. Given a diagram, we
define $C$ as the maximal number of strokes from a bicolored vertex. We use
this number to classify all possible diagrams.

Recall that the $z$-diagram indicates the maximal terms among a finite set
of terms. It is nonempty. If there is a circle, there is an edge of the same
color emanating from it. So there is at least a $z$-stroke, and similarly, at least a
$w$-stroke. \textbf{In each diagram,  the five  vertices are placed  at the 5-th roots of unity, in the counterclockwise order. }

\subsection{No bicolored vertex }

There is at least one isolated edge, which is not a $zw$-edge. Let us say it is a $z$-edge. The complement has 3 bodies. There three can have one or  three $w$-edges according to Rule VI.

For one $w$-edge, the attached bodies have to be $w$-circled by Rule I. This is the first diagram in Figure \ref{fig:C=0}.

For three $w$-edges, the three edges form a triangle. There are three possibilities for the number of  $w$-circled vertices: it is either zero, or two or three (one is not possible by Rule III.)  They constitute the last three diagrams in Figure \ref{fig:C=0}.

\textbf{Hence, we have four possible diagrams, as shown in Figure \ref{fig:C=0}. }

	\begin{figure}[h!]
	
	\centering
	\includegraphics[width=0.8\textwidth]{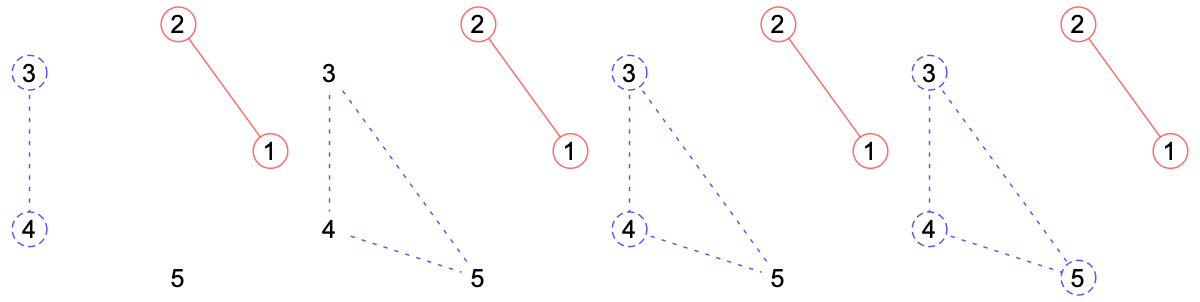}
	
	\caption{Four possible diagrams for no bicolored vertex. }
	\label{fig:C=0}
\end{figure}

\subsection{$C=2$}

There are two cases: a $zw$-edge exists or not.

If it is present,  it is isolated. Let us say, vertex $\textbf{1}$ and vertex $\textbf{2}$ are connected by one $zw$-edge.    Among vertices $\textbf{3}$, $\textbf{4}$, and $\textbf{5}$, there exist both  $z$ and $w$-circle. If none of the three vertices is $z$-circled, we have $\Gamma_1+\Gamma_2=0$ by Rule IV. On the other hand, since vertices $\textbf{3}$, $\textbf{4}$, and $\textbf{5}$ are not $z$-circled, they are not $z$-close to vertex $\textbf{1}$ and vertex $\textbf{2}$ . Then
Proposition \ref{Prp:sumT12} implies that $\Gamma_1+\Gamma_2\ne0$.  This is a contradiction.

Then Rule I implies that there is at least one $z$-stroke and one $w$-stroke among the cluster of vertices $\textbf{3}$, $\textbf{4}$, and $\textbf{5}$. There are two possibilities: whether there is another $zw$-edge or not.

If another $zw$-edge is present, then it is again isolated. This is the first diagram in Figure \ref{fig:C=2}. Note that $w_{15}\approx \epsilon^{-2}$, which  contradicts with Proposition \ref{Prp:dumbbell}, thus impossible.

If another $zw$-edge is not present, there is at least one edge in both color.
 By the circling method, the adjacent vertex is $z$- and $w$-circled.  By the Estimate, the $z$-edge implies the two attached vertices are $w$-close. Then the two ends are both $w$-circled by Rule II. Thus, all three vertices are $z$- and $w$-circled.  Then there are  more strokes in the diagram. This is a contradiction.

If there is no $zw$-edge,  there are adjacent $z$-edges and $w$-edges. From any such adjacency there is no other edge.  Suppose that  vertex $\textbf{1}$ connects with vertex $\textbf{4}$ by $w$-edges and connects with $\textbf{2}$ by $z$-edges. The circling method implies that $\textbf{1}$ is $z$- and $w$-circled, 2 is $w$-circled and 4 is $z$-circled. The color of  $\textbf{2}$ and $\textbf{4}$ forces the color of edges from the circle. If the two edgers go to the same vertex, we get the diagram corresponding to Roberts' continuum at infinity, shown as the second  in Figure \ref{fig:C=2}.

 If the two edgers go to the different  vertices,  the circling method demands a cycle with alternating colors, which is impossible since we have only five edges.

\textbf{Hence,  there is only one  possible  diagram,  the second one in Figure \ref{fig:C=2}.}
\begin{figure}[!h]
	
	\centering
	\includegraphics[width=0.4\textwidth]{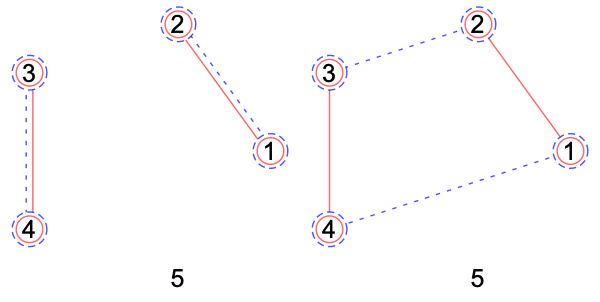}
	
	\caption{Two diagrams for $C=2$. The first one has been excluded.}
	\label{fig:C=2}
\end{figure}

\subsection{$C=3$}

Consider a bicolored vertex  with three strokes. It is vertex $\textbf{1}$ in Figure \ref{fig:C=3},     or connects a single stroke to a $zw$-edge.

We start with the   first  case.  Let us say vertex $\textbf{1}$ connects with vertex $\textbf{2}$ and vertex $\textbf{3}$ by   $z$-edges, and  connects with  vertex $\textbf{4}$ by  a $w$-edge.  There is a $z_{23}$-stroke by Rule VI. The circling method implies that the vertices  $\textbf{1}$, $\textbf{2}$ and vertex $\textbf{3}$ are all $w$-circled, see Figure \ref{fig:C=3}.  Then there is $w$-stroke emanating from $\textbf{2}$ and vertex $\textbf{3}$.  The $w$-stroke may go to vertex $\textbf{4}$, vertex $\textbf{5}$, or it is a  $w_{23}$-stroke.

If  one $w$-stroke goes from vertex $\textbf{2}$ to vertex  $\textbf{4}$, then there is extra $w_{12}$-stroke by Rule VI, which contradict with $C=3$. If all two $w$-strokes go to vertex $\textbf{5}$, then Rule VI implies the existence of $w_{23}$-stroke. This is again a contradiction with $C=3$.

If the $w$-strokes emanating from $\textbf{2}$ and vertex $\textbf{3}$ is just
$w_{23}$-stroke. Then we have an $zw$-edge between $\textbf{2}$, and $\textbf{3}$. Then it is not necessary to discuss the $zw$-edge case.   Then,  we consider the vertex $\textbf{5}$.  It is connected with the previous four vertices or isolated.

If the diagram is connected, vertex $\textbf{5}$   can only connects with vertex $\textbf{4}$  by a $z$-edge (other cases is not possible by Rule VI). Then the circling method implies that vertex $\textbf{5}$   is $w$-circled. Then there is $w$-stroke emanating from $\textbf{5}$. This is a contradiction.

If vertex $\textbf{5}$   is isolated. Then  the circling method implies that all vertices except vertex $\textbf{5}$   are $w$-circled. Only $\textbf{2}$ and vertex $\textbf{3}$ can be $z$-circled, and they are both $z$-circled or both not $z$-circled by Rule IV, see Figure \ref{fig:C=3}.

If both vertex $\textbf{2}$ and vertex $\textbf{3}$ are not $z$-circled, then we have $\Gamma_2+\Gamma_3=0$ and $L_{123}=0$ by Rule IV and Proposition \ref{Prp:LI}. This is a contradiction since
\[  L_{123}= \Gamma_1 (\Gamma_2+\Gamma_3)+ \Gamma_2\Gamma_3.  \]

If both vertex $\textbf{2}$ and vertex $\textbf{3}$ are $z$-circled, then Rule IV implies that
$\Gamma_2+\Gamma_3=0$. On the other hand, Corollary \ref{Cor:sumT12}  implies that
$\Gamma_2+\Gamma_3\ne0$. This is a contradiction.

\textbf{ Thus, the two   diagrams in Figure \ref{fig:C=3} are both excluded. There is  no possible  diagram.}

\begin{figure}[!h]
	
	\centering
	\includegraphics[width=0.4\textwidth]{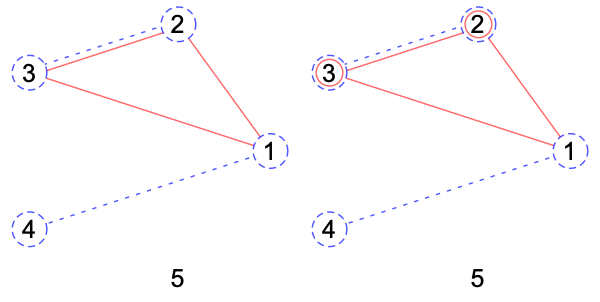}
	
	\caption{Two diagrams for $C=3$. Both have been excluded.}
	\label{fig:C=3}	
\end{figure}


\subsection{$C=4$}

There are five cases: two $zw$-edges,  only one $zw$-edge  and one  edge of each color,    only one $zw$-edge  and two edge of the same color, one  $z$-edge and three $w$-edges, or   two edges of each color emanating from the same vertex.

\subsubsection{} Suppose that there are two $zw$-edges emanating from, e.g.,  vertex  $\textbf{1}$ as on the first diagram in Figure \ref{fig:C=41}.

We get a fully $zw$-edged triangle by Rule VI. This triangle is isolated since
$C=4$.  Since  vertices $\textbf{1}$,  $\textbf{2}$, and $\textbf{3}$ are  $z$-close and $w$-close, if one of them is circled in some color, all of them will be circled in  the same color. Thus,  the first three vertices may be all $z$-circled,  all  $z$-and $w$-circled, or all not circled.

If  vertices $\textbf{1}$,  $\textbf{2}$, and $\textbf{3}$ are  $z$-circled but not $w$-circled, we have $\Gamma_1+\Gamma_2+\Gamma_3=0$ and $L_{123}=0$ by Rule IV and Proposition \ref{Prp:LI}. This is a contradiction since $(\Gamma_1+\Gamma_2+\Gamma_3)^2-2L_{123}\ne 0$. Then the first three vertices can only be all  $z$-and $w$-circled, or all not circled.

The other two vertices can be disconnected, connected by one , e.g.,  $z$-edge, or by one $zw$-edge. Then there are six possibilities, according to whether the first  three vertices are   all  $z$-and $w$-circled, or all not circled, and the connection between the other two vertices. 

Suppose that vertex $\textbf{4}$  and vertex $\textbf{5}$   are connected by one $zw$-edge, and the first three vertices are all not circled. 
Then we have $\Gamma_4+\Gamma_5=0$ by Rule IV. On the other hand, 
Corollary \ref{Cor:sumT12}  implies that $\Gamma_4+\Gamma_5\ne0$.  This is a contradiction.

Then we have the five possibilities, see Figure \ref{fig:C=41}.   Note that Proposition \ref{Prp:triangle}  can further exclude the second and fourth one.

\begin{figure}[!h]
	
	\centering
	\includegraphics[width=0.8\textwidth]{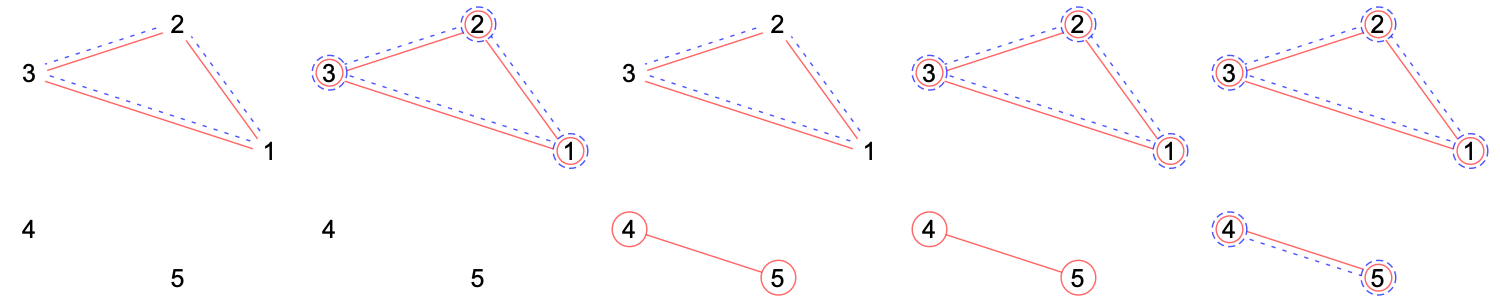}

	\caption{Five diagrams for $C=4$, two $zw$-edges. The second and fourth one has been excluded.}
	\label{fig:C=41}
\end{figure}
\textbf{Hence,  there are three  possible  diagrams,  the first, third and fifth one in Figure \ref{fig:C=41}.}

\subsubsection{}
 Suppose that there are one $zw$-edges and one edge of each color emanating from vertex $\textbf{1}$, as in the first digram of Figure \ref{fig:C=42}. We complete the triangles  by Rule VI.  Note that no more strokes can emanating from vertex $\textbf{1}$ and vertex $\textbf{4}$  since $C=4$. If there are more strokes from vertex $\textbf{2}$,  it can not goes to vertex $\textbf{3}$, since it implies one more stroke emanating from vertex $\textbf{1}$.  Similarly, there is no $z_{35}$-stroke or $w_{25}$-stroke. Then between vertex $\textbf{5}$   and the first four vertices, there can have no edge, one
$z_{25}$-stroke, or one $z_{25}$-stroke and one $w_{35}$-stroke.

For the disconnected diagram,  vertex $\textbf{1}$ and vertex $\textbf{4}$  can not be circled. Otherwise, the circling method implies  either vertex $\textbf{2}$ is be $z$-circled or 3 is $w$-circled,   there should be stroke emanating from vertex $\textbf{2}$ or vertex $\textbf{3}$ . This is a contradiction.  Then vertex $\textbf{2}$ can not be circled, otherwise, vertex $\textbf{3}$  is of the same color by Rule IV. This is a contradiction. Hence, there is no circle in the diagram and this is the first diagram in Figure \ref{fig:C=42}.

If there is only $z_{25}$-stroke, then the circling method implies that  vertices $\textbf{1}$,  $\textbf{2}$, $\textbf{4}$, and $\textbf{5}$ are $z$-circled. Note that vertex $\textbf{3}$  and vertex $\textbf{5}$   can not be $w$-circled, otherwise there are extra $w$-strokes. Then  vertices $\textbf{1}$,  $\textbf{2}$, and $\textbf{4}$ are not $w$-circled
by the circling method.  Note that vertex $\textbf{3}$ must be $z$-circled, otherwise, we have $L_{124}=0$ and $\Gamma_1+\Gamma_4=0$. This is a contradiction.
Then we have the second diagram in Figure \ref{fig:C=42}.

If there are  $z_{25}$-stroke and  $w_{35}$-stroke. Then vertex $\textbf{5}$   is  $z$-and $w$-circled. By the circling method, all vertices are   $z$-and $w$-circled. This the third diagram in Figure  \ref{fig:C=42}.
\begin{figure}[!h]
	
	\centering
	\includegraphics[width=0.8\textwidth]{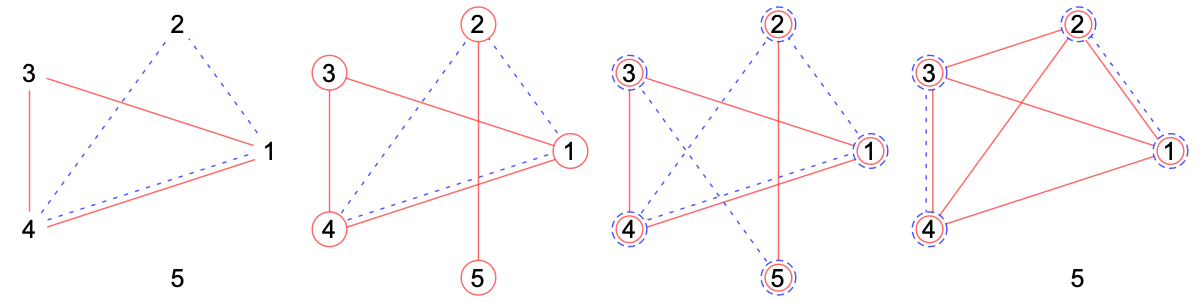}
	
	\caption{Four diagrams for $C=4$, two $zw$-edges. The  fourth one has been excluded.}
	\label{fig:C=42}
\end{figure}

\textbf{Hence,  there are three  possible  diagrams,  the first three  in Figure \ref{fig:C=42}.}

\subsubsection{}

Suppose that there are one $zw$-edges and two $z$-edges  emanating from vertex $\textbf{1}$, as in the  fourth diagram  in  Figure \ref{fig:C=42}. Then there are  $z_{23}$-stroke, $z_{24}$-stroke  and $z_{34}$-stroke by Rule VI.  Note that vertex $\textbf{1}$ is $z$-circled, then the circling method implies all vertices except possibly vertex $\textbf{5}$   are $w$-circled. Then there is $w$-stroke emanating from $\textbf{3}$ and vertex $\textbf{4}$.  The $w$-stroke may go to vertex $\textbf{5}$,  or it is a  $w_{34}$-stroke.

If  all two  $w$-strokes go to vertex $\textbf{5}$, then Rule VI implies the existence of $w_{34}$-stroke, which contradicts with $C=4$.  Then the $w$-strokes from $\textbf{3}$ and vertex $\textbf{4}$ is $w_{34}$-stroke, and  vertex $\textbf{5}$   is disconnected.

Consider the circling the the diagram. We have three different cases: all vertices are not $z$-circled, only two of the first four vertices, say, vertex $\textbf{1}$ and vertex $\textbf{2}$ are $z$-circled, or all the first four vertices are $z$-circled.

If none of the first four vertices is $z$-circled, then  we have $\Gamma_{1234}=0$ and $L_{1234}=0$ by Rule IV and Proposition \ref{Prp:LI}.  This is a contradiction.

If only vertex $\textbf{1}$ and vertex $\textbf{2}$ are $z$-circled,  we have $\Gamma_1+\Gamma_2=0$ by Rule IV. On the other hand, Corollary \ref{Cor:sumT12}  implies that $\Gamma_1+\Gamma_2\ne0$.  This is a contradiction.

Then all the first four vertices are $z$-circled. This gives the   fourth diagram  in  Figure \ref{fig:C=42}. However, this contradicts with Proposition \ref{Prp:dumbbell},  so excluded. 

\textbf{Hence,  there is no  possible  diagram, i.e.,  the fourth one  in Figure \ref{fig:C=42} is impossible. }

\subsubsection{}
Suppose that from vertex $\textbf{1}$ there are three  $w$-edges go to vertex $\textbf{2}$,   vertex $\textbf{3}$  and vertex $\textbf{4}$  respectively and one $z$-edges goes to vertex $\textbf{5}$  .  There is a $w$-stroke between any pair of $\{ \textbf{1, 2, 3, 4} \}$ by Rule VI, and the four vertices are all $z$-circled.

Then there are $z$-strokes emanating from vertex $\textbf{2}$,   vertex $\textbf{3}$  and vertex $\textbf{4}$ .
None of them can go to vertex $\textbf{5}$   by Rule VI. Then there are $z_{23}$-, $z_{24}$- and $z_{34}$-strokes. This contradict with $C=4$. \textbf{Hence, there is no possible diagram in this case.}

\begin{figure}[!h]
	
	\centering
	\includegraphics[width=0.8\textwidth]{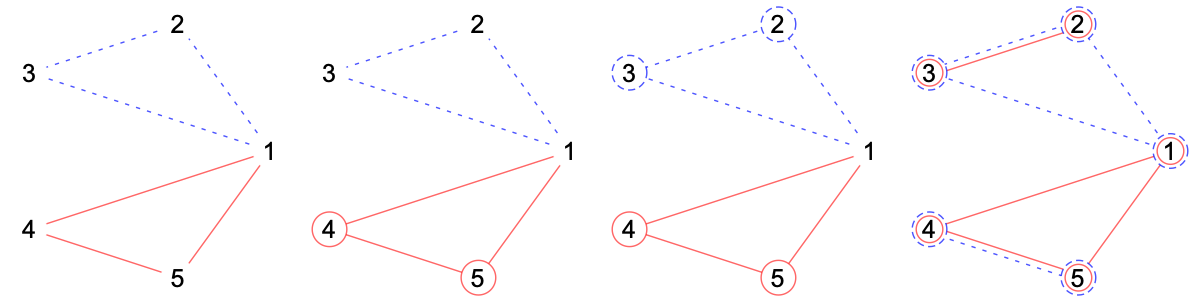}

	\caption{Four diagrams for $C=4$, no $zw$-edges. }
	\label{fig:C=43}
\end{figure}

\subsubsection{}
Suppose that there are two  $w$-edges and two $z$-edges  emanating from vertex $\textbf{1}$, with numeration as in the first diagram of Figure \ref{fig:C=43}.
By Rule VI, there is $w_{23}$- and $z_{45}$-stroke. Thus, we have two attached triangles. By Rule VI, if there is more stroke, it must be $z_{23}$- and/or $w_{45}$-stroke. \\

{\textbf{Case I}}: If there is no more stroke, vertex $\textbf{2}$ and vertex $\textbf{3}$  can only be $w$-circled, and
 vertex $\textbf{4}$  and vertex $\textbf{5}$   can only be $z$-circled. Then vertex $\textbf{1}$ can not be circled, otherwise, the circling method would lead to contradiction.  Then we have the  first three  diagrams in Figure \ref{fig:C=43}. 

{\textbf{Case II}}:  If there is  only $z_{23}$-stroke, the circling method implies vertex $\textbf{1}$, vertex $\textbf{2}$ and vertex $\textbf{3}$  are $z$-circled. Note that only vertex $\textbf{2}$ and vertex $\textbf{3}$  can be $w$-circled. There are two possibilities: whether vertex $\textbf{2}$ and vertex $\textbf{3}$  are both  $w$-circled or both  not $w$-circled.

 If vertex $\textbf{2}$ and vertex $\textbf{3}$  are both  $w$-circled, we  have $\Gamma_2+\Gamma_3=0$ by Rule IV. On the other hand, since  vertices $\textbf{1}$,  $\textbf{4}$, and $\textbf{5}$ are not $w$-circled, they are not $w$-close to vertex $\textbf{2}$ and vertex $\textbf{3}$ . Then
 Proposition \ref{Prp:sumT12} implies that $\Gamma_2+\Gamma_3\ne0$.  This is a contradiction.

  If vertex $\textbf{2}$ and vertex $\textbf{3}$  are both not  $w$-circled, we  have $\Gamma_2+\Gamma_3=0$  and $L_{123}=0$ by Rule IV and Proposition \ref{Prp:LI}. This is a contradiction. Hence there is no possible diagram in this case.

{\textbf{Case III}}: If there are  $z_{23}$- and $w_{45}$-strokes, the circling method implies vertex $\textbf{2}$ and vertex $\textbf{3}$  are $z$-circled, vertex $\textbf{4}$  and vertex $\textbf{5}$   are $w$-circled, and vertex $\textbf{1}$ is  $z$- and $w$-circled. 
By Rule I,  vertex $\textbf{2}$ and vertex $\textbf{3}$ are both $z$-circled, and vertex $\textbf{4}$ and vertex $\textbf{5}$ are both $w$-circled.  This is  the last diagram  in Figure \ref{fig:C=43}. 

\textbf{Hence, there are  four possible diagrams in this case, as shown in Figure \ref{fig:C=43}.}\\





\textbf{ In summary,  among the thirteen diagrams in Figure \ref{fig:C=41}, \ref{fig:C=42} and \ref{fig:C=43}, we have excluded the first and third one in Figure \ref{fig:C=41} and the fourth one in Figure \ref{fig:C=42}.   We have ten possible  diagrams.}

\subsection{$C=5$}

There are three  cases: two $zw$-edges,  one $zw$-edge with one $z$-edge and two  $w$-edges,    one $zw$-edge with three $z$-edges.

\subsubsection{}
Suppose that there are two $zw$-edges and one $z$-edges  emanating from vertex $\textbf{1}$, with numeration as in the first digram of Figure \ref{fig:C=51}.  Rule VI implies the  existence of $zw_{23}$-, $z_{24}$- and $z_{34}$-edges. Then vertex $\textbf{5}$   can be either disconnected  or connects with vertex $\textbf{4}$  by one  $w$-edge, otherwise, it would contradict with $C=5$.

For the disconnected diagram, note that any of the connected four vertices can not be $w$-circled. Otherwise, the circling method implies vertex $\textbf{4}$  is $w$-circled, which is a contradiction.  There are three cases: none of the vertices are circled,  all four vertices except vertex $\textbf{4}$  are $z$-circled, or all four vertices are $z$-circled.

If none of the vertices are circled,  by Proposition \ref{Prp:LI} we have   $L_{123}=0$ since  vertices $\textbf{1}$,  $\textbf{2}$, and $\textbf{3}$ form a triangle with no $w$-circled attached. Similarly, we have $L_{1234}=0$ by Proposition \ref{Prp:LI}. This is a contradiction since
\begin{equation*}
L_{1234}=L_{123}+\Gamma_4(\Gamma_1+\Gamma_2+\Gamma_3), \  (\Gamma_1+\Gamma_2+\Gamma_3)^2-2L_{123}\ne 0.
\end{equation*}

If only  vertices $\textbf{1}$,  $\textbf{2}$, and $\textbf{3}$ are $z$-circled, we also have  $L_{123}$. By Rule IV, we have $\Gamma_1+\Gamma_2+\Gamma_3=0$. This is a contradiction.

Then we have only one possible diagram for the disconnected diagram, and it is the first in Figure \ref{fig:C=51}.

For the connected diagram with $w_{45}$-edge, the circling method implies that all vertices are $w$-circled. Note that vertex $\textbf{4}$  and vertex $\textbf{5}$   can not be $z$-circled. Then we have two cases: all the five vertices are not $z$-circled, or  vertices $\textbf{1}$,  $\textbf{2}$, and $\textbf{3}$ are $z$-circled.

If all the five vertices are not $z$-circled, we have $\Gamma_1+\Gamma_2+\Gamma_3=0$ by Rule IV and $L_{1234}=0$ by Proposition \ref{Prp:LI}. This is a contradiction. 

Then, we have only one possible diagram for the connected diagram, and it is the second in Figure \ref{fig:C=51}.

\begin{figure}[!h]
	
	\centering
	\includegraphics[width=0.4\textwidth]{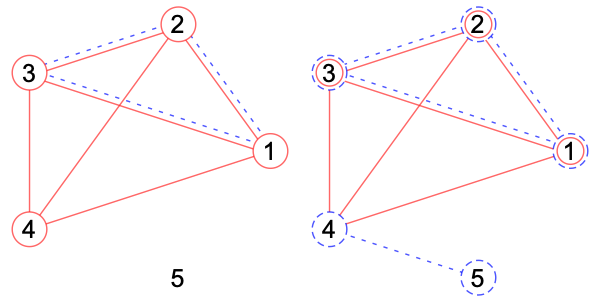}

	\caption{Two diagrams for $C=5$, two $zw$-edges. }
	\label{fig:C=51}
\end{figure}

\textbf{Hence, we have two possible diagrams, shown in Figure \ref{fig:C=51}. }

\subsubsection{}

Suppose that there are one $zw$-edge, one $z$-edge, and two  $w$-edges emanating from vertex $\textbf{1}$, with numeration as in the first diagram of Figure \ref{fig:C=52}.  Rule VI implies that  vertices $\textbf{1}$,  $\textbf{2}$, $\textbf{3}$, and $\textbf{4}$ are fully $w$-stroked, and that
 vertices $\textbf{1}$,  $\textbf{3}$, and $\textbf{5}$ are fully $z$-stroked. There is no more stroke emanating from  vertices $\textbf{1}$,  $\textbf{3}$, and $\textbf{5}$, since that would contradict with $C=5$. There are two cases, $z_{24}$-stroke is present or not.

If it is not present,  vertex $\textbf{2}$ and vertex $\textbf{4}$  can only be $w$-circled, and vertex $\textbf{5}$   can only be $z$-circled. Then vertex $\textbf{1}$ and vertex $\textbf{3}$   can not be circled, otherwise, the circling method would lead to contradiction.  Then we have the first two diagrams in Figure \ref{fig:C=52} by Rule IV.

If $z_{24}$-stroke is present, then  vertices $\textbf{1}$,  $\textbf{2}$, $\textbf{3}$, and $\textbf{4}$ are all $z$-circled. Vertex $\textbf{1}$ and vertex $\textbf{3}$  can not be $w$-circled, which would lead to vertex $\textbf{5}$   also $w$-circled and a contradiction. Vertex $\textbf{2}$ and vertex $\textbf{4}$  are $w$-close, so they are both $w$-circled or both not $w$-circled. If they are both $w$-circled,  then Rule IV implies that
$\Gamma_2+\Gamma_4=0$. On the other hand, Proposition \ref{Prp:sumT12} implies that
$\Gamma_2+\Gamma_4\ne0$. This is a contradiction.

Then according to whether vertex $\textbf{5}$   is $z$-circled or not, we have two cases, which are the last two diagrams in
Figure  \ref{fig:C=52}.

The third diagram in Figure  \ref{fig:C=52} is impossible.  We would have $L_{1234}=0$ and $\sum_{j=1}^4\Gamma_j=0$ by Proposition \ref{Prp:LI} and Rule IV.  Hence, we have only one  possible diagram if  $z_{24}$-stroke is present, and it is the last  diagram in Figure \ref{fig:C=52}.

\begin{figure}[h!]
	
	\centering
	\includegraphics[width=0.8\textwidth]{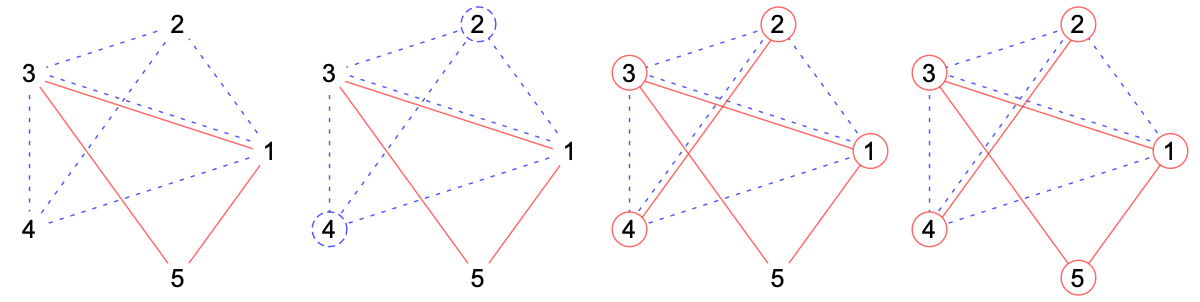}
	
		\caption{Four diagrams for $C=5$, one $zw$-edges. The  third one has been excluded.}
	\label{fig:C=52}
\end{figure}

\textbf{Hence, we have three possible diagrams, the first two and the last one in  Figure \ref{fig:C=52}. }

\subsubsection{}

Suppose that there is one $zw$-edge and three  $w$-edges emanating from vertex $\textbf{1}$ .  Let us say, the $zw$-edge goes to vertex $\textbf{2}$ and the other edges go to the other three vertices.  Rule VI implies that  vertices $\textbf{1, 2, 3, 4, 5}$ are fully $w$-stroked. The circling method implies that all vertices are $z$-circled. Then there are $z$-strokes emanating from vertices $\textbf{3}$, $\textbf{4}$, and $\textbf{5}$. Since $C=5$ at vertex $\textbf{1}$ and vertex $\textbf{2}$,  the $z$-strokes from vertices $\textbf{3}$, $\textbf{4}$, and $\textbf{5}$ must go to vertices $\textbf{3}$, $\textbf{4}$, and $\textbf{5}$. By Rule VI, they form a triangle of $z$-strokes. This is a contradiction with $C=5$.  \textbf{Hence, there is no possible diagram in this case.}  \\

\textbf{In summary,  among the six  diagrams in Figure \ref{fig:C=51} and \ref{fig:C=52}, we have excluded the third one in Figure \ref{fig:C=52}.   We have five possible  diagrams.}

\subsection{$C=6$}

There are three  cases: three  $zw$-edges,  two $zw$-edge with one edge in each color,   two $zw$-edge with two $z$-edge.

\begin{figure}[!h]

	\centering
	\includegraphics[width=0.4\textwidth]{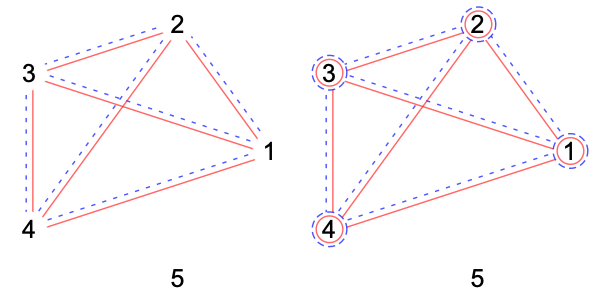}

	\caption{Two diagrams for $C=6$, three $zw$-edges. The  second one has been excluded.}
	\label{fig:C=61}
\end{figure}

\subsubsection{}
Suppose that there are three $zw$-edges  emanating from vertex $\textbf{1}$, with numeration as in the first digram of Figure \ref{fig:C=61}.  Rule VI implies that the  vertices $\textbf{1}$,  $\textbf{2}$, $\textbf{3}$, and $\textbf{4}$ are fully $zw$-edged. Then vertex $\textbf{5}$   must be disconnected  since $C=6$. The first four vertices can be circled in the same way by the circling method.

 If all the first four vertices are $z$-circled but not $w$-circled, then we would have
 $L_{1234}=0$ and $\sum_{j=1}^4\Gamma_j=0$ by Proposition \ref{Prp:LI} and Rule IV.  This is one contradiction.

Then we have two possible diagrams, as shown in Figure \ref{fig:C=61}. However, the second one is impossible by Proposition \ref{Prp:quadrilateral}.

\textbf{Hence, there is only one  possible diagram if   there are three $zw$-edges, the first one   in Figure \ref{fig:C=61}.}

\subsubsection{}
Suppose that there are two $zw$-edges and two $z$-edges emanating from vertex $\textbf{1}$, with numeration as in the first diagram of Figure \ref{fig:C=62}.  Rule VI implies the existence of $z_{25}$-, $z_{35}$-,  $z_{24}$-, $z_{45}$- and $z_{34}$-strokes,   and  that the  vertices $\textbf{1}$,  $\textbf{2}$, and $\textbf{3}$ are fully $zw$-edged. If there is more stroke, it must be $w_{45}$-stroke since $C=6$.

If there is no more stroke, then the vertices can only be $z$-circled. There are two cases, either  vertices $\textbf{1}$,  $\textbf{2}$, and $\textbf{3}$ are $z$-circled or not.

If  vertices $\textbf{1}$,  $\textbf{2}$, and $\textbf{3}$ are $z$-circled, then either vertex $\textbf{4}$  or vertex $\textbf{5}$   must also be $z$-circled. Otherwise, we have $\Gamma_1+\Gamma_2+\Gamma_3=0$ and $L_{123}=0$ by  Proposition \ref{Prp:LI} and Rule IV.  This is one contradiction. Then we have the first two diagrams in Figure \ref{fig:C=62}.

If  vertices $\textbf{1}$,  $\textbf{2}$, and $\textbf{3}$ are not  $z$-circled, then vertex $\textbf{4}$  and vertex $\textbf{5}$   are both $z$-circled or both not $z$-circled. Then we have the second two  diagrams in Figure \ref{fig:C=62}.

If $w_{45}$-stroke is present, then the circling method implies that all vertices are $w$-circled. Rule IV implies that $\Gamma_{123}=0$. If none of vertices $\textbf{1}$,  $\textbf{2}$, and $\textbf{3}$ are  $z$-circled,  Proposition \ref{Prp:LI} yields $L_{123}=0$, a contradiction. Then, at least one, hence all three vertices $\textbf{1}$,  $\textbf{2}$, and $\textbf{3}$ are  $z$-circled. 
Then there are two cases, according to whether vertices $\textbf{4}$ and $\textbf{5}$ are  $z$-circled.  Hence,  we have the last  two  diagrams in Figure \ref{fig:C=62}.

\begin{figure}[!h]

	\centering
	\includegraphics[width=0.6\textwidth]{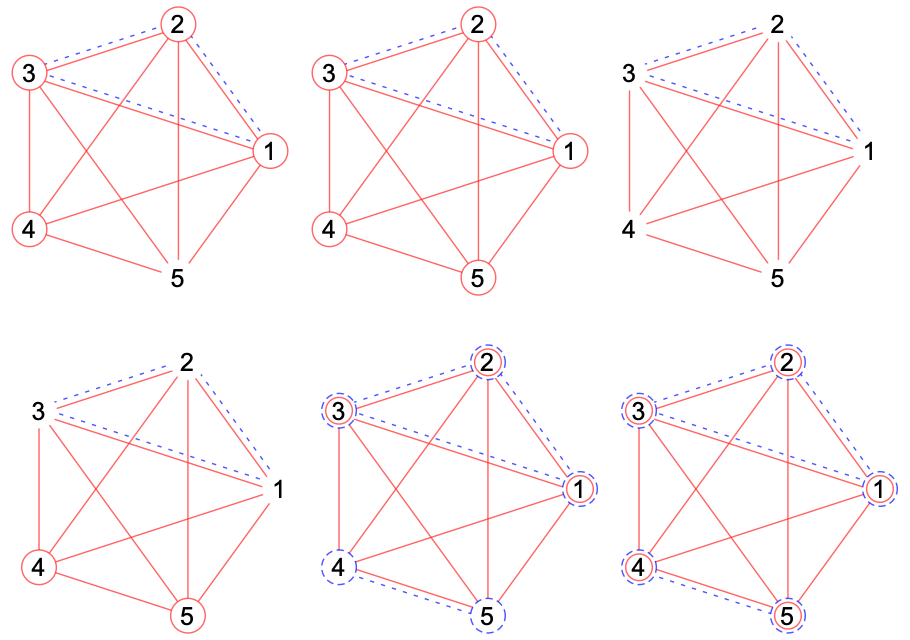}

	\caption{Six diagrams for $C=6$, two  $zw$-edges, case 1. }
	\label{fig:C=62}
\end{figure}

\textbf{Hence, there are six possible diagrams in this case, as shown in Figure \ref{fig:C=62}. }

\subsubsection{}
Suppose that there are two $zw$-edges and one edge of each color  emanating from vertex $\textbf{1}$, with numeration as in the first diagram of Figure \ref{fig:C=63}.  Rule VI implies the existence of $w_{25}$-, $w_{35}$-,  $z_{24}$-,  and $z_{34}$-strokes,   and  that the  vertices $\textbf{1}$,  $\textbf{2}$, and $\textbf{3}$ are fully $zw$-edged. If there is more stroke, it must be $w_{45}$-stroke, but it would lead to extra stroke emanating from vertex $\textbf{1}$, which contradicts with $C=6$.

Vertex $\textbf{4}$  can only be $w$-circled, and vertex $\textbf{4}$  can only be $z$-circled. Then  vertices $\textbf{1}$,  2 and  3  can not be circled, otherwise, the circling method would lead to contradiction.  \textbf{Then we have the  diagram in Figure \ref{fig:C=63} by Rule IV. }

\begin{figure}[h!]

	\centering
	\includegraphics[width=0.2\textwidth]{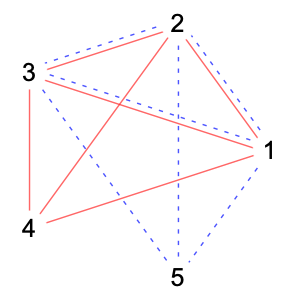}

	\caption{One diagram for $C=6$, two  $zw$-edges, case 2.}
	\label{fig:C=63}
\end{figure}

\textbf{In summary,   among the nine diagrams of  Figure \ref{fig:C=61}, \ref{fig:C=62},   and \ref{fig:C=63}, we exclude the second one of 
Figure \ref{fig:C=61}. That is, we have eight possible diagrams. }

\subsection{$C=7$}
Suppose that there are three $zw$-edges and one $z$-edge  emanating from vertex $\textbf{1}$, with numeration as in the  diagram of Figure \ref{fig:C=7}.  Rule VI implies that the  vertices $\textbf{1}$,  $\textbf{2}$, $\textbf{3}$, and $\textbf{4}$ are fully $zw$-edged and that  vertex $\textbf{5}$   connects with the first four vertices by one $z$-edge. There is no more stroke since $C=7$. If any vertex is $w$-circled, all are $w$-circled, which would lead to $w$-stroke emanating from vertex $\textbf{5}$   . This is a contradiction.  There are two cases, either vertex $\textbf{1}$ is $z$-circled or not.

If vertex $\textbf{1}$ is not $z$-circled, then none of the vertices is $z$-colored by the circling method. Then $L_{1234}=0$ and $L=0$ by  Proposition \ref{Prp:LI}, a contradiction.

If vertex $\textbf{1}$ is $z$-circled,  then  vertices $\textbf{1}$,  $\textbf{2}$, $\textbf{3}$, and $\textbf{4}$ are all $z$-circled. In this case,  vertex $\textbf{5}$   must be  $z$-circled. Otherwise, 
we have  $L_{1234}=0$ and $\sum_{j=1}^4\Gamma_j=0$ by  Proposition \ref{Prp:LI} and Rule IV,  a contradiction.

\textbf{Hence, we only have one diagram in the case of $C=7$, as in Figure \ref{fig:C=7}.}

\begin{figure}[!h]

\centering
\includegraphics[width=0.2\textwidth]{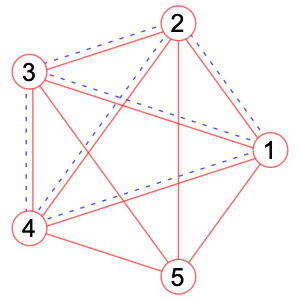}

	\caption{One diagram for $C=7$. }
	\label{fig:C=7}
\end{figure}


\subsection{$C=8$}
Suppose that there are four $zw$-edges   emanating from vertex $\textbf{1}$ . Rule VI implies that the  vertices $\textbf{1, 2, 3, 4, 5}$ are fully $zw$-edged. Since all vertices are both $z$-close and $w$-close,the vertices are all $z$- and $w$-circled, or all not circled. If they are all just $z$-circled but not $w$-circled, then we have 
\[  \Gamma=0, L=0,  \]
by Rule IV and Proposition \ref{Prp:LI}, a contradiction.

\begin{figure}[!h]

	\centering
	\includegraphics[width=0.4\textwidth]{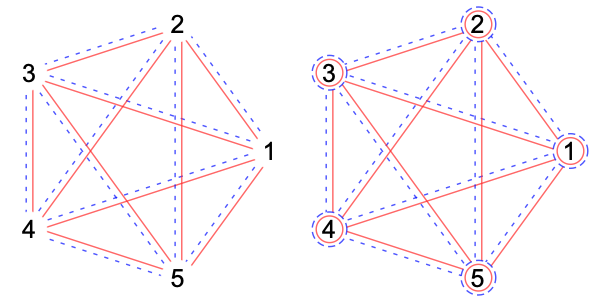}

	\caption{Two diagrams for $C=8$. }
	\label{fig:C=8}
\end{figure}

\textbf{Hence, we  have two  diagrams in the case of $C=8$, as in Figure \ref{fig:C=8}. } \\

\section{Conclusion}

\label{sec:diagram&constraints}

\indent\par 

In the searching for all possible two-colored diagrams of Sect. \ref{sec:matrixrules}, we have found 39 of them, as shown in Figure \ref{fig:C=0}-\ref{fig:C=8}. \textbf{Among them, we have  excluded eight of them}, i.e., the first diagram in Figure \ref{fig:C=2}, the two diagrams in Figure \ref{fig:C=3}, the second and fourth diagram in Figure \ref{fig:C=41}, the fourth diagram  in Figure \ref{fig:C=42}, the third diagram in Figure \ref{fig:C=52}, and
the second diagram in Figure \ref{fig:C=61}. 

\textbf{The conclusion  is that any singular sequence should converge to one of the remaining  thirty-one diagrams}. We further divide  them in two lists,  see Figure \ref{fig:list1}  and  \ref{fig:list2}. 
The first list contains  \textbf{nine diagrams}  and the second list contains  \textbf{ 22 diagrams}, where these diagrams are ordered not
by the number of strokes from a bicolored vertex, but by the number of circles.  Notably, the first list of nine diagrams can be 
excluded by some further arguments. We will report this in a future work.

\begin{figure}[!h]

	\centering
	\includegraphics[width=0.9\textwidth]{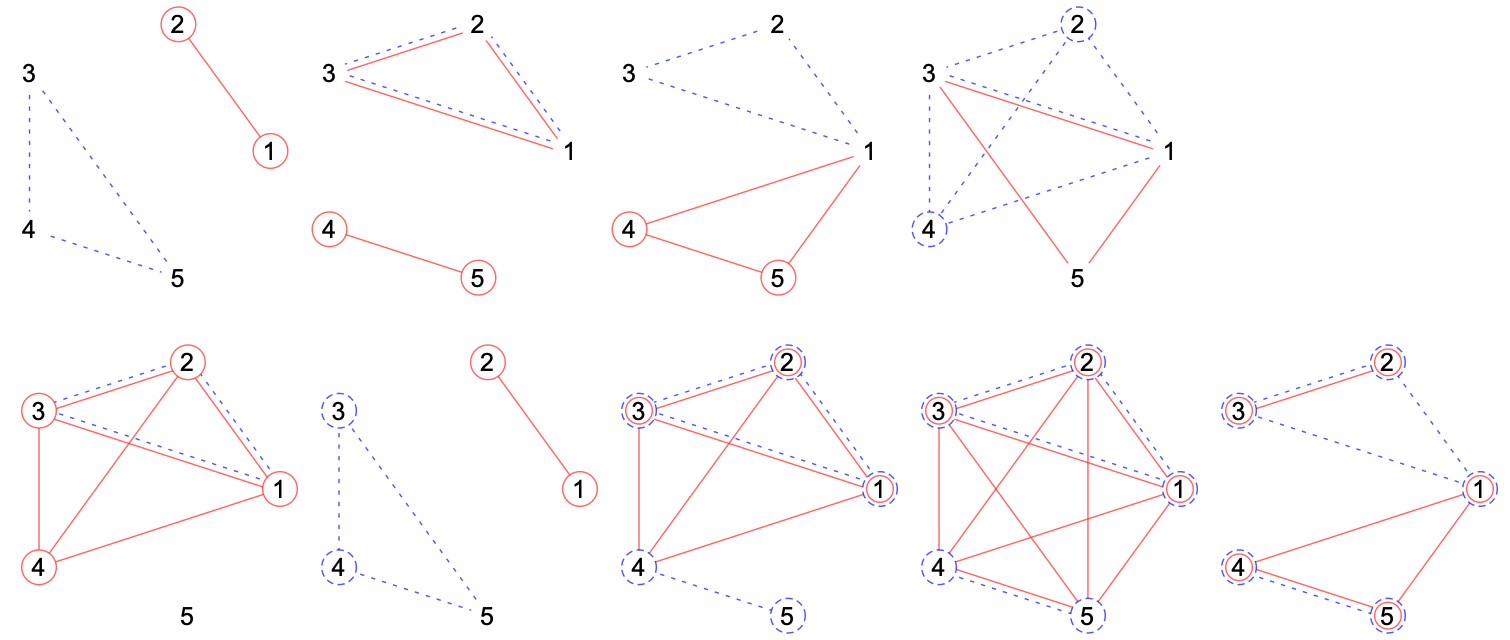}

	\caption{The nine  impossible diagrams.}
	\label{fig:list1}
\end{figure}

\begin{figure}[!h]

	\centering
	\includegraphics[width=0.9\textwidth]{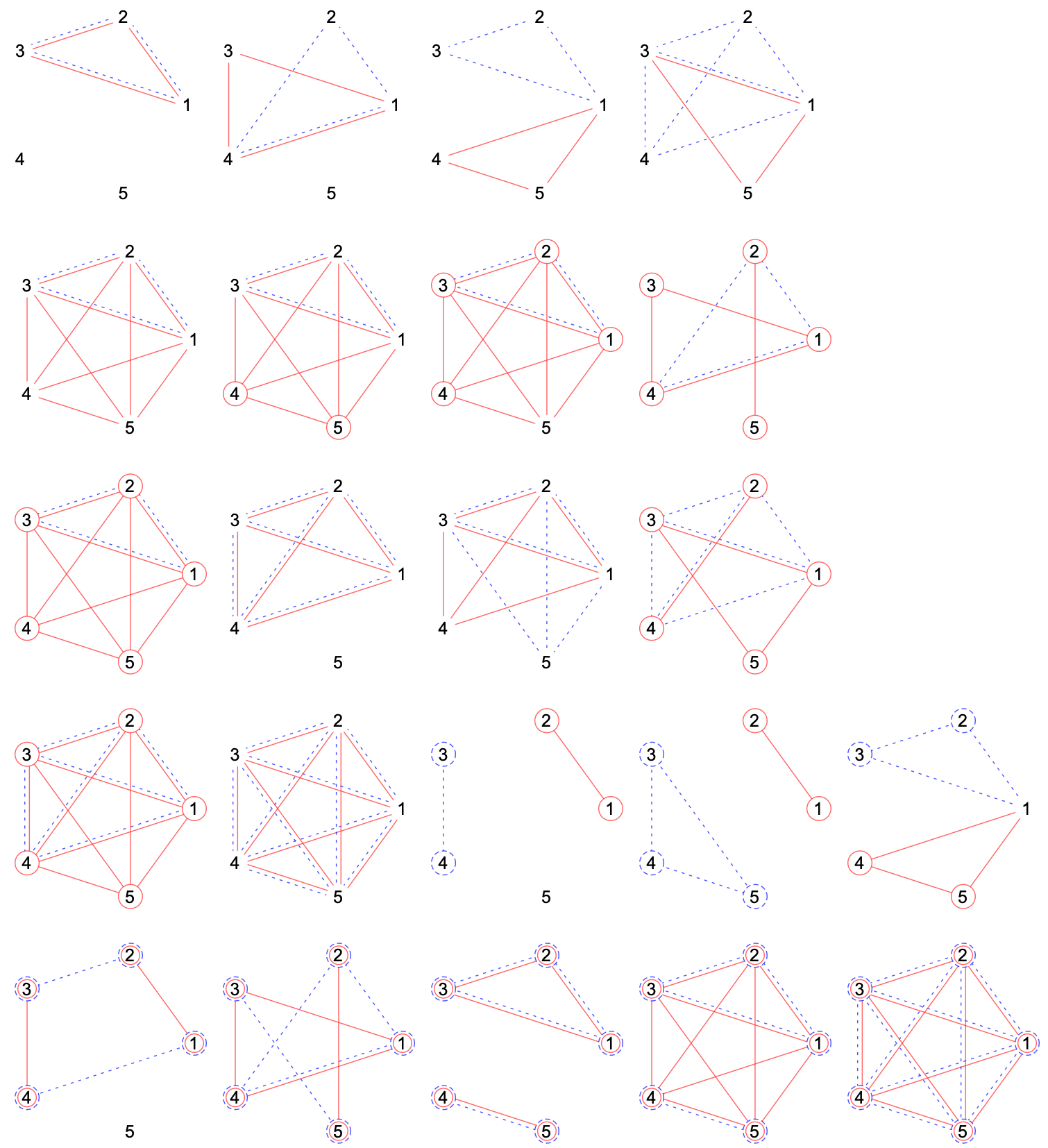}

	\caption{The 22 possible diagrams.  }
	\label{fig:list2}
\end{figure}


\newpage

\begin{thebibliography}{10}
	
	\bibitem{Albouy2012Finiteness}
	Alain Albouy and Vadim Kaloshin.
	\newblock Finiteness of central configurations of five bodies in the plane.
	\newblock {\em Annals of Mathematics}, 176(1):535--588, 2012.
	

	
	\bibitem{aref2003vortex}
	H.~Aref, P.~K. Newton, M.~A. Stremler, T.~Tokieda, and D.~L Vainchtein.
	\newblock Vortex crystals.
	\newblock {\em Advances in applied Mechanics}, 39:2--81, 2003.
	

	
	
	
	
	

	
	
	
	
	
	
	
	\bibitem{hampton2009finiteness}
	M.~Hampton and R.~Moeckel.
	\newblock Finiteness of stationary configurations of the four-vortex problem.
	\newblock {\em Transactions of the American Mathematical Society},
	361(3):1317--1332, 2009.
	

	
	
	
	\bibitem{helmholtz1858integrale}
	H.~Helmholtz.
	\newblock {\"U}ber integrale der hydrodynamischen gleichungen, welche den
	wirbelbewegungen entsprechen.
	\newblock {\em Journal f{\"u}r die reine und angewandte Mathematik},
	1858(55):25--55, 1858.
	
	
	
	\bibitem {Newton2001}
	P.~K. Newton. 
	\newblock {\em The {$N$}-vortex problem,  Analytical techniques}. 
	\newblock Applied Mathematical Sciences 145, Springer-Verlag, New York, 2001. 
	
	
	

	

	
	\bibitem{o1987stationary}
	K.~A. O'Neil.
	\newblock Stationary configurations of point vortices.
	\newblock {\em Transactions of the American Mathematical Society},
	302(2):383--425, 1987.
	

	
	
	
	
	
	
	
	
	
	
	
	
	
	
	
	
	
	

	
	
	
	

	
	
	
	\bibitem{yu2021Finiteness}
	Xiang Yu.
	\newblock Finiteness of stationary configurations of the planar four-vortex
	problem.
	\newblock { \em Advances in Mathematics} (2023), \url{ https://doi.org/10.1016/j.aim.2023.109378}
	
	

	
	
	
\end{thebibliography}
\end{document}